\documentclass[12pt]{article}

\usepackage{amsmath, amssymb, amsthm}
\usepackage{xcolor}

\usepackage{graphicx}
\usepackage{hyperref}
\usepackage{geometry}

\usepackage{pgfplots}
\pgfplotsset{compat=1.18}

\newtheorem{theorem}{Theorem}[section]
\newtheorem{lemma}[theorem]{Lemma}
\theoremstyle{definition}
\newtheorem{definition}[theorem]{Definition}

\newtheorem{proposition}[theorem]{Proposition}
\theoremstyle{remark}
\newtheorem{remark}[theorem]{{\bf Remark}}

\DeclareMathOperator{\sign}{\mathrm{sign}}
\newcommand{\dt}{\partial_t} 
\newcommand{\dx}{\partial_x}

\numberwithin{equation}{section}

\title{The traveling wave solutions of the 1D hyperbolic Keller-Segel equations}
\author{Xin Liu\footnote{\href{mailto:xliu23@tamu.edu}{xliu23@tamu.edu}, Department of Mathematics, Texas A\&M University, College Station, TX 77843-3368}, William Kyle Barker \footnote{\href{mailto:wkbarker@ualr.edu}{wkbarker@ualr.edu}, Department of Mathematics, University of Arkansas at Little Rock, 2801 S. University Ave., Little Rock, AR 72204}}
\date{\today}

\begin{document}

\maketitle

\begin{abstract}
The goal of this paper is to investigate the traveling wave solutions with one-sided far field conditions for the one-dimensional hyperbolic Keller-Segel equations with quorum sensitivity. The traveling wave solutions are piece-wise smooth in the moving coordinate, and satisfy the entropy inequality at the discontinuity.

\smallskip 


\end{abstract}

\section{Introduction}

\subsection{The hyperbolic Keller-Segel system with quorum sensitivity}

Let $ \sigma \in [0,1] $ be the density of bacteria
and $ S $ be the chemical potential. The hyperbolic Keller-Segel system with logistic/quorum sensitivity is given by the following system of PDEs, \cite{perthameExistenceSolutionsHyperbolic2009,leeThresholdShockFormation2015}:
\begin{subequations}
    \label{sys:hyperbolicKS}
    \begin{align}
        \label{eq:continuity} \partial_t \sigma + \partial_x (\sigma (1-\sigma) \partial_x S) & = 0, \\
        \label{eq:potential} - \partial_{xx} S + S & = \sigma.
    \end{align}
\end{subequations}

The goal of this work is to investigate the non-trivial stationary and traveling wave solutions to system \eqref{sys:hyperbolicKS} with either one-sided far field vacuum state,
\begin{equation}
    \label{far-field-vacuum}
    S (-\infty) = \sigma(-\infty) = 0. 
\end{equation}
or one-sided non-vacuum far field,
\begin{equation}
    \label{non-vacuum-far-field}
    S(-\infty) = \sigma(-\infty) = \sigma_\infty \in (0,1]. 
\end{equation}

Notice that, thanks to the symmetry $ x \rightarrow -x $ of system \eqref{sys:hyperbolicKS}, \eqref{far-field-vacuum} or \eqref{non-vacuum-far-field} is equivalent to the far field vacuum/non-vacuum condition at $ + \infty $. In addition, $ \sigma \equiv S \equiv 0 $ (or correspondingly, $ \sigma \equiv S \equiv \sigma_\infty $) is a trivial solution. The existence of non-trivial stationary and traveling wave solutions implies the non-uniqueness of solutions to the problem. 

\subsection{Backgrounds and motivations}

The (Patlak-)Keller-Segel system was introduced in \cite{patlakRandomWalkPersistence1953,kellerInitiationSlimeMold1970}, to describe the collective motion of cells that are attracted/repelled by a self-emitted chemical substance. We refer to \cite{horstmann1970PresentKellerSegel2003,hillenUsersGuidePDE2009} for historic reviews and literature. 

There are a large amount of literature focusing on the parabolic-elliptic or parabolic-parabolic KS system of production type. There exists a critical total mass, such that solutions blow up if and only if the initial total mass is in the supercritical regime, where the diffusion is not strong enough to counterbalance the growth. See \cite{eganafernandezUniquenessLongTime2016,blanchetTwodimensionalKellerSegelModel2006,blanchetInfiniteTimeAggregation2008} for the study in two space dimensions. Similar results for the degenerate KS system can be found in \cite{sugiyamaGlobalExistenceSubcritical2006,sugiyamaApplicationBestConstant2007,blanchetCriticalMassPatlak2009}.

A conditional global stability result was obtained in 
\cite{feireislConvergenceEquilibriaKeller2007}. Recently, with small initial data, \cite{hsiehLongtimeDynamicsClassical2024} shows the global stability of the constant state. 

On the other hand, the blow up profile for the KS system in three space dimensions is first obtained in
\cite{soupletBlowupProfilesParabolic2019} and later in
\cite{baiBlowupProfileKeller2025}. See  \cite{liuFiniteTimeBlowup2025} for the blow up profile for the Keller-Segel-Navier-Stokes system. The study of weak solutions in the scaling invariant class can be found in \cite{kozonoExistenceUniquenessTheorem2012}. We refer to 
\cite{hieberStrongSolutionsKellerSegelNavierStokes2025,naGlobalWellposednessTwodimensional2024, chaeExistenceSmoothSolutions2012,winklerGlobalLargeDataSolutions2012,winklerStabilizationTwodimensionalChemotaxisNavierStokes2014,ahnStabilityInstabilityFully2026,tanTimePeriodicStrong2021} for the study of the KS system coupled with viscous fluids. 

For the hyperbolic(-elliptic) KS system of production type, the solution will blow up in finite time \cite{liuLargeFrictionLimit2026}. However, taking into account the logistic/quorum sensitivity as in \eqref{sys:hyperbolicKS}, the solution (if exists) will remain bounded. In \cite{perthameExistenceSolutionsHyperbolic2009}, 
a global weak solution was constructed for the KS system with logistic/quorum sensitivity. In general, however, a smooth solution will form a shock in finite time for such a system, as shown in  \cite{leeThresholdShockFormation2015, mengGlobalWellposednessBlowup2024,naFinitetimeBlowupHyperbolic2024} for both production and consumption types. Notably, the blow up is not the consequence of vacuum, or large (in the Sobolev norm) data, but purely the shock formation similar to Burgers' equation.

In particular, the result in \cite{perthameExistenceSolutionsHyperbolic2009} has shown that the entropy condition is not enough to select a unique global weak solution. 

Our goal of this paper is to construct the traveling wave solutions to system \eqref{sys:hyperbolicKS} as a class of global weak solutions. The traveling wave solutions we seek are piece-wise smooth, and satisfy the entropy condition \eqref{eq:entropy_inequality}, as in \cite{perthameExistenceSolutionsHyperbolic2009}. In particular, we construct a class of global weak solutions to system \eqref{sys:hyperbolicKS}. 

The remainder of this paper is organized as follows. In Section 2, we derive the entropy inequality and the corresponding Rankine–Hugoniot jump conditions for admissible discontinuities. Section 3 is devoted to the construction of stationary solutions with vacuum far field, where we analyze the possible types of jumps and establish the existence of a one-parameter family of piecewise smooth profiles. In Section 4, we construct traveling wave solutions with vacuum far field, treating both the continuous profiles and those with a single discontinuity, and we provide a complete classification of the possible wave speeds and jump configurations. The case of non-vacuum far field is considered in Section 5, where we first classify the stationary solutions and then show that all traveling wave profiles are necessarily continuous, giving their full characterization. Finally, Section 6 contains a summary of our results and conclusion.

\subsection{Main results}

\begin{theorem}[Stationary solutions]
\label{thm:stn-sol}
\begin{enumerate}
    \item (Vacuum far field) 
For any fixed constant $ A_0 \leq 0 $, 
there exists a class of stationary solutions to system \eqref{sys:hyperbolicKS} with \eqref{far-field-vacuum} satisfying the entropy inequality \eqref{eq:entropy_inequality}, below, centered at $ x_0 \in \mathbb R $, given by
\begin{subequations}\label{stn-sol}
\begin{equation}
    \label{eq:stn-first-kind}
    (\sigma_s(x), S_s(x)) = \begin{cases}
        (0 , A_0 e^x), & x \leq x_0, \\
        (1, 1 + C_0 e^x + D_0 e^{-x}), & x_0 < x , 
    \end{cases}
\end{equation}
with 
\begin{equation}
    \label{eq:stn-24}
    D_0 e^{-x_0} = - \frac{1}{2}, \qquad C_0 e^{x_0} = A_0 e^{x_0}- \frac{1}{2}, \qquad \forall \ A_0 \leq 0. 
\end{equation}
\end{subequations}

\item (Non-vacuum far field) For $ \sigma_\infty \in (0,1) $, system \eqref{sys:hyperbolicKS} with \eqref{non-vacuum-far-field} admits a class of stationary solutions satisfying the entropy inequality \eqref{eq:entropy_inequality}, below, given by 
either
\begin{itemize}
    \item \begin{equation}
    \label{eq:stn-nv-105}
    (\sigma_s(x),S_s(x)) = \begin{cases}
        (\sigma_\infty,\sigma_{\infty}), & x \leq x_0, \\
        (0,A_0 e^x + B_0 e^{-x}), & x_0 < x , 
    \end{cases} 
\end{equation}
with 
\begin{equation}
    \label{eq:stn-nv-106}
    \sigma_\infty \in (0,1),\qquad A_0 e^{x_0} = B_0 e^{-x_0} = \frac{1}{2} \sigma_\infty, \qquad A_0, \ B_0 > 0;
\end{equation} or 
\item  \begin{equation}
    \label{eq:stn-nv-205}
    (\sigma_s(x),S_s(x)) = \begin{cases}
        (\sigma_\infty,\sigma_{\infty}), & x \leq x_0, \\
        (1,1 + C_0 e^x + D_0 e^{-x}), & x_0 < x , 
    \end{cases} 
\end{equation}
with 
\begin{equation}
    \label{eq:stn-nv-206}
    \qquad \sigma_\infty \in  (0,1), \qquad C_0 e^{x_0} = D_0 e^{-x_0} = \frac{\sigma_\infty - 1}{2}, \qquad C_0, \ D_0 < 0.  
\end{equation}
\end{itemize}

\end{enumerate}
\end{theorem}

\begin{theorem}[Traveling wave solutions]
\label{thm:tvl-sol}
There exist traveling wave solutions
\begin{equation}
    \sigma(t,x)=\sigma_c(\xi),
    \qquad
    S(t,x)=S_c(\xi),
    \qquad
    \xi=x-ct,
\end{equation}
satisfying the entropy inequality \eqref{eq:entropy_inequality}, below, for all $ c $, to system \eqref{sys:hyperbolicKS} with either \eqref{far-field-vacuum} or \eqref{non-vacuum-far-field}. 
\begin{enumerate}
\item 
In the case of with vacuum far field \eqref{far-field-vacuum}, the right-going ($c > 0$) and the left-going ($ c< 0 $) traveling wave solutions can be either continuous (given in Propositions \ref{prop:tw-cnt} and \ref{prop:lg-cnt-tw}, below), or piece-wise smooth with only one discontinuity/jump (given in Propositions \ref{prop:rg-jump-tw} and \ref{prop:lg-jump-tw}, below). 
    \item In the case of with non-vacuum far field \eqref{non-vacuum-far-field}, $ \sigma_\infty \in (0,1) $, the traveling wave profile can only be continuous (as shown in Proposition \ref{prop:tw-nv-cnt}, below), given by Propositions \ref{prop:tw++}, \ref{prop:tw+-}, \ref{prop:tw-+}, and \ref{prop:tw--}, below. 
\end{enumerate}
\begin{remark}
\label{remark:thm}
    The stationary solutions and the traveling wave solutions for the saturated far field, i.e., $ \sigma_\infty = 1 $, are similar to the case of vacuum far field, i.e., $ \sigma_\infty = 0 $, by applying the change of variables $ \eta := 1-\sigma $ and $ Q:= 1-S $. See section \ref{subsec:saturation-far-field}, below.
\end{remark}

\end{theorem}

\section{The entropy inequality and the jump conditions}
\label{sec:entropy}

For a smooth solution to system \eqref{sys:hyperbolicKS}, the entropy equality is calculated by multiplying \eqref{eq:continuity} with $ \eta'(\sigma) $, for any convex entropy $ \eta = \eta(\sigma) $. This yields 
\begin{equation}
    \label{eq:entropy_equality}
    \partial_t \eta(\sigma) + \partial_x ( \partial_x S q(\sigma) ) + (\sigma-S) (q - g \eta')(\sigma) = 0,  
\end{equation}
where the logistic sensitivity function $ g = g(\sigma) $ is defined as
\begin{equation}
    \label{def:g} 
    g(\sigma):= \sigma (1-\sigma),
\end{equation} 
and $ q(\sigma) $ is the `entropy flux' satisfying
\begin{equation}
    \label{def:q}
    q'(\sigma) = g'(\sigma) \eta'(\sigma)= (\eta(\sigma) - 2\sigma \eta(\sigma))' + 2 \eta(\sigma).
\end{equation}
For weak solutions, one expects the following entropy inequality:
\begin{equation}
    \label{eq:entropy_inequality}
    \partial_t \eta(\sigma) + \partial_x ( \partial_x S q(\sigma) ) + (\sigma-S) (q - g \eta')(\sigma) \leq 0.
\end{equation}

\smallskip 

The goal of this section is to derive the jump conditions across the discontinuity. 
Consider a solution with a discontinuity in the space-time $ \mathbb R \times \mathbb R $ given by $ \lbrace (t,x) = (t, x(t)) \in \mathbb R \times \mathbb R \rbrace $. Then the space-time normal vector to the discontinuity is given by 
\begin{equation}
    \label{def:normal}
    \vec \nu := \frac{1}{\sqrt{1 + \vert x'(t) \vert^2}} \begin{pmatrix}
        x'(t) \\ -1
    \end{pmatrix} 
\end{equation}
pointing to the negative direction in space. In addition, for any function $ f = f(t,x) $, the left and right values of $ f $ at the discontinuity are denoted by
\begin{equation}
    \label{def:l-r-value}
    f^-(t) := \lim_{x \to x(t)^-} f(t,x), \qquad \text{and} \qquad f^+(t) := \lim_{x \to x(t)^+} f(t,x),
\end{equation}
respectively, and the jump across the discontinuity is given by 
\begin{equation}
    \label{def:jump}
    [[f]] := f^+ - f^-. 
\end{equation}

In this section, our goal is to derive the following proposition:
\begin{proposition}[Jump conditions]
    \label{prop:jump-entropy}
    Let $ \lbrace (t,x) = (t,x(t)) \in \mathbb R \times \mathbb R \rbrace $ be a space-time curve of the discontinuity of the solution to system \eqref{sys:hyperbolicKS}. Then at the discontinuity curve, one has that
    \begin{align}
        [[S]] = [[\dx S]] & = 0, \label{eq:jump-1} \\
        x'(t) & = \dx S(1- \sigma^+ - \sigma^-), \label{eq:jump-2} \qquad \text{and} \\
        \label{eq:jump-3}& \begin{cases}
            \dx S \leq 0 & \text{if} \ \sigma^+ > \sigma^-,\\
            \dx S \geq 0 & \text{if} \ \sigma^+ < \sigma^-.
        \end{cases}
    \end{align}
\end{proposition}

\begin{proof}

\smallskip 

The Rankine-Hugoniot condition for system \eqref{sys:hyperbolicKS} reads
\begin{align}
    \label{eq:RH-1}
    x'(t) [[\sigma]] - [[\sigma (1-\sigma) \partial_x S]] & = 0 \qquad \text{on} ~ \lbrace (t,x) = (t, x(t)) \rbrace, \\
    \label{eq:RH-2}
    [[\dx S]] & = 0 \qquad \text{on} ~ \lbrace (t,x) = (t, x(t)) \rbrace.
\end{align}
Similarly, the Rankine-Hugoniot condition for the entropy inequality \eqref{eq:entropy_inequality} reads
\begin{equation}
    \label{eq:RH-3}
    x'(t)[[\eta(\sigma)]]  - [[ \dx S q(\sigma)]] \leq  0 \qquad \text{on} ~ \lbrace (t,x) = (t,x(t)) \rbrace. 
\end{equation}

\smallskip

Then one can calculate from \eqref{eq:RH-1}--\eqref{eq:RH-3} that
\begin{align}
    [[S]] = [[\dx S]] & = 0, \label{eq:jump-1-0} \\
    x'(t) & = \dx S(1- \sigma^+ - \sigma^-), \label{eq:jump-2-0} \qquad \text{and} \\
    \dx S (1-\sigma^+ - \sigma^- )[[\eta(\sigma)]] & \leq  \dx S [[q(\sigma)]]. \label{eq:jump-3-0}
\end{align}
In particular, by taking Kru\v{z}kov's absolute value entropy \cite{kruzkovFIRSTORDERQUASILINEAR1970,bouchutKruzkovEstimatesScalar1998}, for $ \forall\ \sigma_0 \in [0,1] $,
\begin{equation}
\label{entropy-abs}
\eta_{\sigma_0}(\sigma) := \vert \sigma-\sigma_0 \vert,
\end{equation}
one can calculate from \eqref{def:q} that 
\begin{equation}
    \label{entropy-flux-abs}
    q_{\sigma_0}(\sigma):= \vert \sigma - \sigma_0 \vert - 2 \sigma \vert \sigma - \sigma_0\vert + \vert \sigma - \sigma_0\vert^2 \sign(\sigma-\sigma_0) + C
\end{equation}
for some constant $ C \in (0,\infty) $. 

\smallskip 

In particular, for $ \sigma_0 \geq \max\lbrace \sigma^+, \sigma^- \rbrace $ or $ \sigma_0 \leq \min\lbrace \sigma^+, \sigma^- \rbrace $, \eqref{eq:jump-3} is trivial for the Kru\v{z}kov's entropy. For $ \sigma^- \leq \sigma_0 \leq \sigma^+ $, one has that 
\begin{align*}
    [[\eta_{\sigma_0}(\sigma)]] &= (\sigma^+-\sigma_0) - (\sigma_0 - \sigma^-) = \sigma^+ + \sigma^- - 2 \sigma_0, \\
    [[q_{\sigma_0}(\sigma)]] & = \lbrack (\sigma^+ - \sigma_0) - 2 \sigma^+ ( \sigma^+ - \sigma_0) + (\sigma^+ - \sigma_0)^2 \rbrack \\
    & \qquad - \lbrack (\sigma_0 - \sigma^-) - 2 \sigma^- ( \sigma_0 - \sigma^- ) - (\sigma^- - \sigma_0)^2 \rbrack \\
    & = \sigma^+ + \sigma^- - (\sigma^+)^2 - (\sigma^-)^2 + 2 \sigma_0^2 - 2 \sigma_0.
\end{align*}
Thus \eqref{eq:jump-3} implies
\begin{equation}
\label{eq:jump-4}
    \dx S (\sigma_0 - \sigma^+) (\sigma_0 - \sigma^-) 
    \geq 0.
\end{equation}
Similarly, for $ \sigma^+ \leq \sigma_0 \leq \sigma^- $, one has that 
\begin{equation}
\label{eq:jump-5}
    \dx S (\sigma_0 - \sigma^+) (\sigma_0 - \sigma^-) 
    \leq 0.
\end{equation}
This finishes the proof of Proposition \ref{prop:jump-entropy}. 

\end{proof}

\section{Stationary solutions with vacuum far field}
\label{sec:stationary}

In this section, we are looking for stationary solutions to \eqref{sys:hyperbolicKS}. That is, we consider $ \sigma_s= \sigma_s(x) $ such that 
\begin{subequations}
    \label{sys:stationary}
    \begin{align}
        \label{eq:stn-01} \dx(\sigma_s (1-\sigma_s) \dx S_s) & = 0, \\
        \label{eq:stn-02} - \partial_{xx} S_s + S_s & = \sigma_s.  
    \end{align}
\end{subequations}
Substituting \eqref{eq:stn-02} into \eqref{eq:stn-01} and integrating in $ x $ yield
\begin{equation}
    \label{eq:stn-03}
    (-\partial_{xx} S_s + S_s)(1+\partial_{xx} S_s - S_s) \partial_x S_s = 0, 
\end{equation}
where we have taken that 
\begin{equation}
    \label{eq:stn-04}
    S_s(-\infty) = 0. 
\end{equation}

Then one of the following three conditions must hold at any point $ x \in \mathbb R $:
\begin{align}
    - \partial_{xx} S_s + S_s & = 0, \label{eq:stn-05} \\
    1 + \partial_{xx} S_s - S_s & =0, \label{eq:stn-06} \\
    \partial_x S_s & =0. \label{eq:stn-07}
\end{align}
Correspondingly, $ S_s $ must take the form of 
\begin{align}
    \text{Type I:} \quad && S_s(x) & = A e^x + B e^{-x}, && \qquad (\sigma_s = 0, \text{correspondingly})\label{eq:stn-08} \\
    \text{Type II:} \quad && S_s(x) &= 1 + C e^x + D e^{-x}, && \qquad (\sigma_s = 1, \text{correspondingly})\label{eq:stn-09} \\
    \text{or} \qquad \text{Type III:} \quad && S_s(x) &= E, && \qquad (\sigma_s = E, \text{correspondingly}) \label{eq:stn-10}
\end{align}
for some constants $ A,B,C,D,E \in \mathbb R $, with $ E \in [0,1] $. In general, there are six types of jumps between these three types of solutions. In the following, we will analyze each type of jump. 




Let $ -\infty < x_0 < x_1 < x_2 < \cdots < x_n < \infty $ be the discontinuity points/jumps of $ \sigma_s $.
Thanks to \eqref{eq:stn-04}, it is only possible to have Type I \eqref{eq:stn-08} or Type III \eqref{eq:stn-10} for $ x < x_0 $. 

\subsection{Stationary solution starting with Type I at $ -\infty $}
\label{sec:stationary_first_kind}

Thanks to \eqref{eq:jump-1} and \eqref{eq:stn-04}, it is only possible to have that a jump from Type I \eqref{eq:stn-08} to Type II \eqref{eq:stn-09} at $ x_0 $: 
\begin{equation}
    \label{eq:stn-11}
    S_s(x) = \begin{cases}
        A_0 e^{x}, & x \leq x_0, \\
        1 + C_0 e^x + D_0 e^{-x}, & x_0 < x \leq x_1, 
    \end{cases}
\end{equation}
with, thanks to \eqref{eq:jump-1},
\begin{equation}\label{eq:stn-12}
    A_0 e^{x_0} = C_0 e^{x_0} - D_0 e^{-x_0} = 1 + C_0 e^{x_0} + D_0 e^{-x_0}. 
\end{equation}
Hence 
\begin{equation}
    \label{eq:stn-14}
    D_0 e^{-x_0} = - \frac{1}{2}, \qquad C_0 e^{x_0} = A_0 e^{x_0} - \frac{1}{2},
\end{equation}
and 
\begin{equation}
    \label{eq:stn-15}
    \sigma_s(x_0^-) = 0 < \sigma_s(x_0^+) = 1. 
\end{equation}
Therefore, \eqref{eq:jump-3} implies that $ \dx S_s(x_0) \leq 0 $, and thus
\begin{equation}
    \label{eq:stn-17}
    A_0 \leq 0.
\end{equation}
Therefore, $ C_0 < 0 $ and $ D_0 < 0 $. 

\smallskip 

We claim that at $ x_1 $, it is impossible to have a jump from Type II \eqref{eq:stn-09} to Type III \eqref{eq:stn-10}. Otherwise, at $ x = x_1 $, thanks to \eqref{eq:jump-1}, one has that 
\begin{equation}
    \label{eq:stn-18}
    0 \overset{\eqref{eq:stn-10}}{=} \dx S_s(x_1) = C_0 e^{x_1} - D_0 e^{-x_1} < C_0 e^{x_0} - D_0 e^{-x_0} \overset{\eqref{eq:stn-12}}{=} A_0 e^{x_0} \leq 0,
\end{equation}
since $ C_0, D_0 <0 $, and $ x_1 > x_0 $. This is a contradiction. 

\smallskip 

Hence at $ x = x_1 $, we consider a jump from Type II \eqref{eq:stn-09} to Type I \eqref{eq:stn-08}. That is, 
\begin{equation}
    \label{eq:stn-19}
    S_s(x) = \begin{cases}
        1 + C_0 e^x + D_0 e^{-x}, & x_0 < x \leq x_1, \\
        A_1 e^{x} + B_1 e^{-x}, & x_1 < x \leq x_2, 
    \end{cases}
\end{equation}
with 
\begin{equation}
    \label{eq:stn-20}
    \begin{gathered}
    S_s(x_1) = 1 + C_0 e^{x_1} + D_0 e^{-x_1} =  A_1 e^{x_1} + B_1 e^{-x_1}, \quad \\
    \text{and} \quad \dx S_s(x_1) = C_0 e^{x_1} - D_0 e^{-x_1} =  A_1 e^{x_1} - B_1 e^{-x_1} < 0,
    \end{gathered}
\end{equation}
thanks to \eqref{eq:stn-18}.
In particular, 
\begin{equation}
    \label{eq:stn-21}
    \sigma_s(x_1^-) = 1 > \sigma_s(x_1^+) = 0,
\end{equation}
and the entropy condition \eqref{eq:jump-3} implies that $ \dx S_s(x_1) \geq 0 $, which contradicts \eqref{eq:stn-20}.

\smallskip

Therefore, $ x = x_0 $ is the only jump point and the stationary solution is given by 
\begin{subequations}\label{stn-sol-1}
\begin{equation}
    \label{eq:stn-first-kind}
    (\sigma_s(x), S_s(x)) = \begin{cases}
        (0 , A_0 e^x), & x \leq x_0, \\
        (1, 1 + C_0 e^x + D_0 e^{-x}), & x_0 < x , 
    \end{cases}
\end{equation}
with 
\begin{equation}
    \label{eq:stn-24}
    D_0 e^{-x_0} = - \frac{1}{2}, \qquad C_0 e^{x_0} = A_0 e^{x_0}- \frac{1}{2}, \qquad \forall \ A_0 \leq 0. 
\end{equation}
\end{subequations}

\begin{figure}[htbp]
\centering
\begin{tikzpicture}
\begin{axis}[
    xlabel={$x$},
    ylabel={Value},
    xmin=-4, xmax=4,
    ymin=-1.25, ymax=1.25,
    samples=200,
    grid=both,
    legend pos=south west,
    width=13cm,
    height=5cm
]
\addplot[domain=-5:0, blue, ultra thick] {0};
\addplot[domain=0:2, blue, ultra thick, forget plot] {1 - 0.5*exp(x) - 0.5*exp(-x)};
\addlegendentry{$S_s(x)$}

\addplot[domain=-5:0, red, dashed, ultra thick] {0};
\addplot[domain=0:5, red, dashed, ultra thick, forget plot] {1};
\addlegendentry{$\sigma_s(x)$}

\draw[red, dotted, ultra thick] (axis cs:0,0) -- (axis cs:0,1);

\end{axis}
\end{tikzpicture}
\caption{Solution \eqref{stn-sol-1} for $x_0=0, A_0=0$ or \eqref{stn-sol-2}, below}
\end{figure}
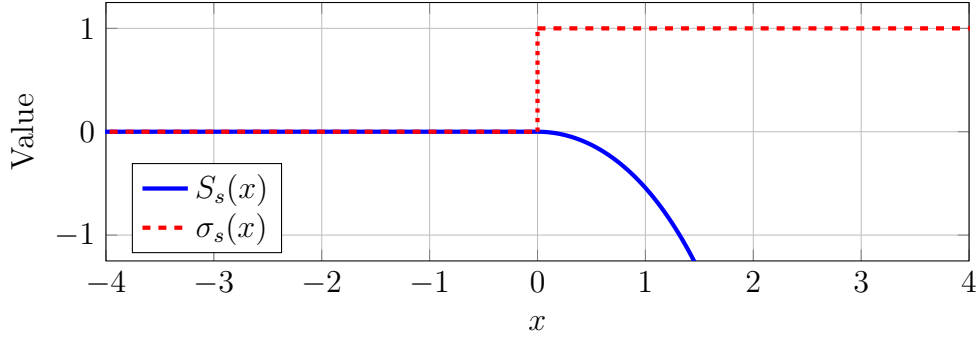

\begin{figure}[htbp]
\centering
\begin{tikzpicture}
\begin{axis}[
    xlabel={$x$},
    ylabel={Value},
    xmin=-4, xmax=4,
    ymin=-1.25, ymax=1.25,
    samples=200,
    grid=both,
    legend pos=south west,
    width=13cm,
    height=5cm
]

\addplot[domain=-5:0, blue, ultra thick] {-exp(x)};
\addplot[domain=0:1, blue, ultra thick, forget plot] {1 - 1.5*exp(x) - 0.5*exp(-x)};
\addlegendentry{$S_s(x)$}

\addplot[domain=-5:0, red, dashed, ultra thick] {0};
\addplot[domain=0:5, red, dashed, ultra thick, forget plot] {1};
\addlegendentry{$\sigma_s(x)$}

\draw[red, dotted, ultra thick] (axis cs:0,0) -- (axis cs:0,1);

\end{axis}
\end{tikzpicture}
\caption{Solution \eqref{stn-sol-1} for $x_0=0, A_0=-1$}
\end{figure}
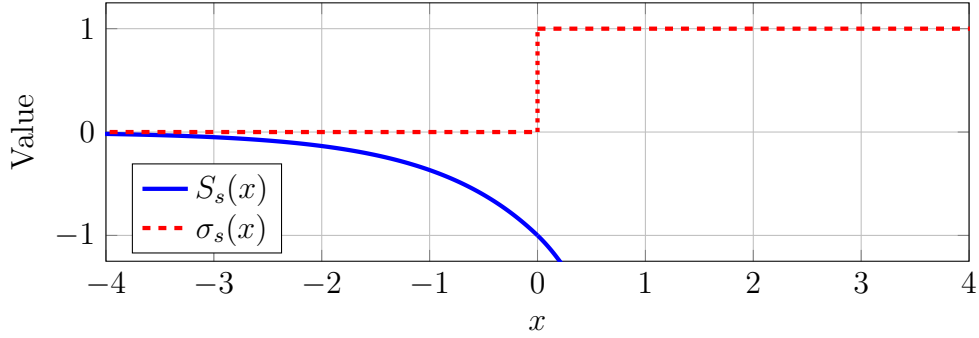

\subsection{Stationary solution starting with Type III at $ - \infty $}
\label{sec:stationary_second_kind}

Thanks to \eqref{eq:jump-1}, \eqref{eq:stn-04}, and \eqref{eq:stn-10}, the only possible Type III solution for $ S_s $ for $ x < x_0 $ is $ S_s = 0 $, i.e., $ E = 0 $. At $ x = x_0 $, $ S_s $ can only jump from Type III \eqref{eq:stn-10} to Type II \eqref{eq:stn-09}:
\begin{equation}
    \label{eq:stn-001}
    S_s(x) = \begin{cases}
        0, & x \leq x_0, \\
        1 + C_0 e^x + D_0 e^{-x}, & x_0 < x \leq x_1.
    \end{cases}
\end{equation}
Similar as before, the jump conditions \eqref{eq:jump-1} yields that 
\begin{equation}
    \label{eq:stn-002}
    1 + C_0 e^{x_0} + D_0 e^{-x_0} = 0 = C_0 e^{x_0} - D_0 e^{-x_0},
\end{equation}
and hence
\begin{equation}
    \label{eq:stn-003}
    C_0 e^{x_0} = D_0 e^{-x_0} = - \frac{1}{2}.
\end{equation}
Moreover, \eqref{eq:stn-001} satisfies the entropy condition \eqref{eq:jump-3}.

\smallskip 

At $ x = x_1 $, we claim that it is impossible to have a jump from Type II \eqref{eq:stn-09} to Type III \eqref{eq:stn-10}. Indeed, thanks to \eqref{eq:stn-003}, we have $ C_0, D_0 <0 $. Since $ x_1 > x_0 $, one has that
\begin{equation}
    \label{eq:stn-004} 
    \dx S_s (x_1) = C_0 e^{x_1} - D_0 e^{-x_1} < C_0 e^{x_0} - D_0 e^{-x_0} = 0. 
\end{equation} 
But any Type III solution has $ \dx S_s = 0 $. This violates the jump condition \eqref{eq:jump-1}. 

\smallskip 

Now we consider a jump from Type II \eqref{eq:stn-09} to Type I \eqref{eq:stn-08} at $ x = x_1 $:
\begin{equation}
    \label{eq:stn-005}
    S_s(x) = \begin{cases}
        1 + C_0 e^x + D_0 e^{-x}, & x_0 < x \leq x_1, \\
        A_1 e^{x} + B_1 e^{-x}, & x_1 < x \leq x_2. 
    \end{cases}
\end{equation}
Then the jump condition \eqref{eq:jump-1} yields that:
\begin{equation}
    \label{eq:stn-006}
    \begin{gathered}
    S_s(x_1) = 1 + C_0 e^{x_1} + D_0 e^{-x_1} = A_1 e^{x_1} + B_1 e^{-x_1}, \quad \\
    \text{and} \quad \dx S_s(x_1) = C_0 e^{x_1} - D_0 e^{-x_1} = A_1 e^{x_1} - B_1 e^{-x_1} < 0.
    \end{gathered}
\end{equation}
Hence, one has that
\begin{equation}
    \label{eq:stn-007}
    A_1 e^{x_1} = \frac{1}{2} + C_0 e^{x_1}, \qquad B_1 e^{-x_1} = \frac{1}{2} + D_0 e^{-x_1}. 
\end{equation}
However, since $ \sigma_s(x_1^-) =1 > \sigma_s(x_1^+) = 0 $, the entropy condition \eqref{eq:jump-3} is violated thanks to \eqref{eq:stn-006}.

\smallskip 

Therefore, $ x = x_0 $ is the only jump point and the stationary solution is given by 
\begin{subequations} \label{stn-sol-2}
\begin{equation}
    \label{eq:stn-second-kind}
    (\sigma_s(x),S_s(x)) = \begin{cases}
        (0,0), & x \leq x_0, \\
        (1,1 + C_0 e^x + D_0 e^{-x}), & x_0 < x , 
    \end{cases}
\end{equation}
for $ C_0, D_0 < 0 $, and 
\begin{equation}
    \label{eq:stn-008}
    C_0 e^{x_0} = D_0 e^{-x_0} = - \frac{1}{2}.
\end{equation}
\end{subequations}
Notice that, this is exactly \eqref{stn-sol-1} with $ A_0 = 0 $. This finishes the proof of theorem \ref{thm:stn-sol}.

\section{Traveling wave solutions with vacuum far field}
\label{sec:traveling_wave}

In this section, we construct traveling-wave solutions for
\eqref{sys:hyperbolicKS} satisfying a vacuum condition at only one spatial
infinity. Unlike the stationary solutions studied in the previous section,
these profiles are genuinely propagating and develop a nontrivial active region connected to vacuum.

\begin{definition}[One-sided traveling semi-wave]
A one-sided traveling semi-wave for \eqref{sys:hyperbolicKS} is a
traveling-wave solution of the form
\begin{equation}
\label{eq:tw-ansatz}
\sigma(t,x)=\sigma_c(\xi),
\qquad
S(t,x)=S_c(\xi),
\qquad
\xi=x-ct,
\end{equation}
such that the profile approaches a prescribed equilibrium state at one spatial
infinity only. More precisely, there exists an equilibrium state
\[
(\sigma_-,S_-)
\]
satisfying
\[
\sigma_c(\xi)\to\sigma_-,
\qquad
S_c(\xi)\to S_-,
\qquad
\partial_\xi S_c(\xi)\to0
\]
as \(\xi\to-\infty\) or as \(\xi\to+\infty\), while no equilibrium condition
is imposed at the opposite spatial infinity.

Consequently, the profile need not converge to a second equilibrium state.
Instead, it may approach a different asymptotic regime, such as unbounded
growth of \(S\) or convergence of \(\sigma\) to a value different from the
prescribed equilibrium.

Such a solution is characterized by the wave speed
\[
c\in\mathbb R
\]
together with the prescribed equilibrium state.
\end{definition}

In the construction below, we focus on the case where the prescribed
equilibrium is vacuum at
$
\xi\to-\infty
$. 

Substituting the ansatz \eqref{eq:tw-ansatz} into
\eqref{sys:hyperbolicKS} yields
\begin{subequations}
    \label{sys:tw}
\begin{align}
-c\partial_\xi\sigma_c
+\partial_\xi\big(\sigma_c(1-\sigma_c)\partial_\xi S_c\big)
&=0,
\label{eq:tw-01}
\\[4pt]
-\partial_{\xi\xi}S_c+S_c
&=\sigma_c.
\label{eq:tw-02}
\end{align}
\end{subequations}
In the present setting, 
we impose the left vacuum condition
\begin{equation}
\sigma_c(\xi)\to0,
\qquad
S_c(\xi)\to0,
\qquad
\text{as }\xi\to-\infty.
\label{eq:tw-left-vacuum}
\end{equation}
No condition is prescribed as \(\xi\to+\infty\); in particular, the profile
need not approach any equilibrium there.

Integrating \eqref{eq:tw-01} once gives
\[
-c\sigma_c+\sigma_c(1-\sigma_c)\partial_\xi S_c=K,
\]
where \(K \in \mathbb R\) is a constant of integration. As \(\xi\to-\infty\), we have
\[
\sigma_c(\xi)\to0,
\qquad
\partial_\xi S_c(\xi)\to0.
\]
Hence
\[
K
=
\lim_{\xi\to-\infty}
\left(
-c\sigma_c+\sigma_c(1-\sigma_c)\partial_\xi S_c
\right)
=
0.
\]
Consequently, one has the zero-flux relation
\begin{equation}
-c\sigma_c+\sigma_c(1-\sigma_c)\partial_\xi S_c=0, \qquad \forall \ \xi \in \mathbb R. 
\label{eq:tw-03}
\end{equation}

\smallskip

Equation \eqref{eq:tw-03} is automatically satisfied in any region where
\[
\sigma_c=0.
\]
In any active region where
\[
\sigma_c>0,
\]
equation \eqref{eq:tw-03} reduces to
\[
-c+(1-\sigma_c)\partial_\xi S_c=0.
\]
This dichotomy between vacuum and active regions is the basis for the
construction below.

\subsection{Jump and far field vacuum state in traveling wave solutions with non-trivial speed}
\label{subsec:no-jump-tw}
If there is jump discontinuity of $ \sigma_c $, 
without loss of generality, assume the jump point is at $ \xi_0 = 0 $. There there are only types of jump that can happen:
\begin{itemize}
    \item Jump from (left or right) vacuum to (right or left) non-vacuum state, referred to as {\bf type A jump};
    \item Jump from non-vacuum state to another non-vacuum state, referred to as {\bf type B jump}. 
\end{itemize}
{\noindent\bf Type A jump in traveling wave solutions with $ c \neq 0 $:} Without loss of generality, assume that $ \xi = 0 $ is the jump location, and 
\begin{equation}
\label{tw-jump-001}
    \sigma_c(0+) > 0, \qquad \sigma_c(0-) = 0. 
\end{equation}
Then from \eqref{eq:tw-03}, one has that
\begin{equation}
\label{tw-jump-002}
    -c + (1-\sigma_c(0+)) \partial_\xi S_c(0)= 0. 
\end{equation}
This is the same as the Rankine-Hugoniot condition \eqref{eq:jump-2}. Meanwhile, the entropy condition \eqref{eq:jump-3} implies that 
\begin{equation}
    \label{tw-jump-102}
    \partial_\xi S_c(0) \leq 0.
\end{equation}
Thus \eqref{tw-jump-002} and \eqref{tw-jump-102} imply that, since $ c \neq 0 $, 
\begin{equation}
    \label{tw-jump-103}
    c = (1-\sigma_c(0+)) \partial_\xi S_c(0) < 0.
\end{equation}

\smallskip 

Following similar arguments, one can conclude that, if $ \sigma_c(0+) = 0 $ and $ \sigma_c(0-) > 0 $, one has
\begin{equation}
    \label{tw-jump-104}
    c = (1-\sigma_c(0-)) \partial_\xi S_c(0) > 0.
\end{equation}

\smallskip

{\noindent\bf No type B jump in traveling wave solutions with $ c \neq 0 $:} Without loss of generality, assume that
\begin{equation}
    \label{tw-jump-004}
    \sigma_c(0+) = a, \qquad \sigma_c(0-) = b, \qquad a \neq b, \ \text{and} \  a, b \in (0,1]. 
\end{equation}
Then from \eqref{eq:tw-03}, one has that
\begin{equation}
\label{tw-jump-005}
    -c + (1-\sigma_c(0+)) \partial_\xi S_c(0) = 0 = -c + (1-\sigma_c(0-)) \partial_\xi S_c(0),
\end{equation}
where, thanks to \eqref{eq:jump-1}, $ S_c(0+) = S_c(0-) = S_c(0) $. However, this implies that 
\begin{equation}
   \partial_\xi S_c(0)= c = 0, 
\end{equation}
which leads to contradiction. 

We have shown the following proposition:
\begin{proposition}
    \label{prop:no-jump-tw}
    For any traveling wave solutions to system \eqref{sys:hyperbolicKS} with non-zero traveling speed $ c \neq 0 $, there are only two possible jumps:
    \begin{itemize}
        \item either $ \sigma_c $ jumps from zero to non-zero, in which case $ c = (1-\sigma_c(\xi_0+))\partial_\xi S_c(\xi_0) < 0 $;
        \item or $ \sigma_c $ jumps from non-zero to zero, in which case $ c = (1-\sigma_c(\xi_0-))\partial_\xi S_c(\xi_0) > 0 $.
    \end{itemize} 
    Here $ \xi_0 $ is the location of the discontinuity/jump. 
\end{proposition}

\smallskip 

Meanwhile, we claim that there will always be vacuum state connecting the left far field; That is:
\begin{lemma}[No asymptotic vacuum]
\label{lm:no-asymptotic-vaccum}
    For a traveling wave solution with non-trivial speed, i.e., $ c \neq 0 $, satisfying \eqref{sys:tw} and \eqref{eq:tw-left-vacuum}, there exists $ \xi_0 \in \mathbb R $ such that
\begin{equation}
    \label{tw:vacuum-001}
    \sigma_c(\xi) = 0 \qquad \forall \ \xi < \xi_0. 
\end{equation}
\end{lemma}
\begin{proof}
Otherwise, there exists $ \xi_1 \in \mathbb R $ such that 
\begin{equation}
    \label{tw:vacuum-002}
    \sigma_c(\xi) \in (0,1] \qquad \forall \ \xi < \xi_1. 
\end{equation}
Then \eqref{eq:tw-03} implies that 
\begin{equation}
    \label{tw:vacuum-003}
    - c + (1-\sigma_c(\xi))\partial_\xi S_c(\xi) = 0 \qquad \forall \ \xi < \xi_1. 
\end{equation}
Taking the limit $ \xi \to -\infty $ in \eqref{tw:vacuum-003} implies that 
\begin{equation}
    c = 0, 
\end{equation}
thanks to \eqref{eq:tw-left-vacuum}, leading to contradiction. 
\end{proof}

\subsection{Right-going traveling wave solutions, i.e., $ c > 0 $}
\label{subsec:right-going-traveling-wave}

\subsubsection{Right-going continuous traveling wave}
\label{subsubsec:rg-continous-tw}

This section focuses on constructing the class of right-going continuous traveling wave solution to \eqref{sys:hyperbolicKS}. In particular, we will show the following proposition:

\begin{proposition}[Right-going continuous traveling wave]
\label{prop:tw-cnt}
Let \(c>0\). Then there exists a continuous global traveling-wave profile of \eqref{sys:hyperbolicKS}, satisfying \eqref{sys:tw}--\eqref{eq:tw-left-vacuum}.
After translation, the profile is given on the left by
\[
\sigma_c(\xi)=0,
\qquad
S_c(\xi)=ce^\xi,
\qquad
\xi\le0.
\]
For \(\xi>0\), the pair \((S_c,Y_c := \partial_\xi S_c)\),
solves the initial value problem
\begin{equation}
\begin{cases}
\partial_\xi S_c=Y_c,\\[4pt]
\partial_\xi Y_c=S_c-1+\dfrac{c}{Y_c},
\end{cases}
\qquad
S_c(0)=c,
\qquad
Y_c(0)=c,
\label{eq:tw-system}
\end{equation}
and
\begin{equation}
\sigma_c(\xi)=1-\frac{c}{Y_c(\xi)}.
\label{eq:tw-sigma}
\end{equation}
Moreover,
\[
Y_c(\xi)>c
\qquad
\text{for all }\xi>0,
\]
and consequently,
\[
0<\sigma_c(\xi)<1
\qquad
\text{for all }\xi>0.
\]

Additionally,
\[
S_c(\xi)\to+\infty,
\qquad
Y_c(\xi)\to+\infty,
\qquad
\sigma_c(\xi)\to1^-
\qquad
\text{as }\xi\to+\infty.
\]
Thus the profile is a one-sided traveling semi-wave.
\end{proposition}

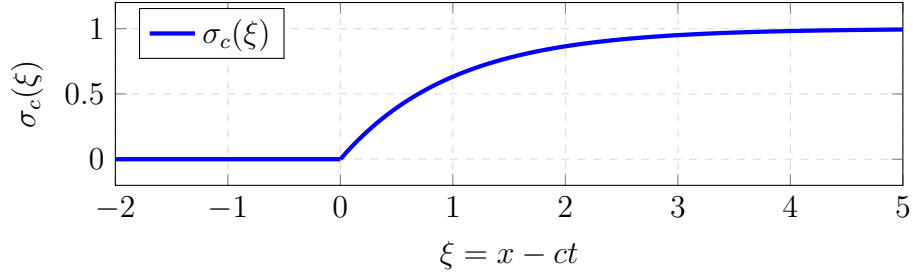
\begin{figure}[htbp]
\centering
\begin{tikzpicture}
\begin{axis} [
    xlabel={$\xi = x - ct$},
    ylabel={$\sigma_c(\xi)$},
    xmin=-2, xmax=5,
    ymin=-0.2, ymax=1.2,
    samples=200,
    grid=both,
    grid style={dashed, gray!30},
    width=12cm,
    height=4cm,
    legend pos=north west
]

\addplot[domain=-2:0, blue, ultra thick] {0};
\addplot[domain=0:6, blue, ultra thick] { 1 - 1 / (exp(x))};
\addlegendentry{$\sigma_c(\xi)$}



\end{axis}
\end{tikzpicture}
\caption{Sample of traveling wave profile $\sigma_c(\xi), \ c > 0 $, from Proposition \ref{prop:tw-cnt}}
\label{fig:sigma_c_traveling_rg-cnt}
\end{figure}

To show proposition \ref{prop:tw-cnt},
thanks to lemma \ref{lm:no-asymptotic-vaccum}, one has \eqref{tw:vacuum-001}. Without loss of generality, we assume that $  \xi_0 = 0 $ and it is the largest number such that $ \sigma_c $ is vaccum on the left. That is,
\begin{equation}
\label{cnt-tw-001}
0 = \sup\lbrace \xi_0 \vert \sigma_c(\xi) \equiv 0, \, \forall \ \xi < \xi_0 \rbrace. 
\end{equation}

Therefore \eqref{eq:tw-02} implies
\begin{equation}\label{cnt-tw-002}
-\partial_{\xi\xi}S_c+S_c=0 \qquad \forall \ \xi < 0.
\end{equation}
Thanks to \eqref{eq:tw-left-vacuum}, one can solve from \eqref{cnt-tw-002} that 
\begin{equation}\label{cnt-tw-003}
S_c(\xi)=Ae^\xi,
\qquad
\partial_\xi S_c(\xi)=Ae^\xi \qquad \forall \ \xi < 0,
\end{equation}
for some $ A \in \mathbb R $. 

Meanwhile, there exists $ \xi_1 > 0 $ such that 
\begin{equation}
    \label{cnt-tw-004}
    \sigma_c(\xi) > 0 \qquad \forall \ 0 < \xi < \xi_1,
\end{equation}
and therefore from \eqref{eq:tw-03} implies that
\begin{equation}
    \label{cnt-tw-005}
    -c + (1-\sigma_c) \partial_\xi S_c = 0 \qquad \forall \ 0 < \xi < \xi_1.
\end{equation}
In particular, \eqref{cnt-tw-005} implies that
\begin{equation}
    \label{cnt-tw-006}
    \partial_\xi S_c(0) = c,
\end{equation}
since $ \sigma_c $ is continuous at $ \xi = 0 $. Together with \eqref{cnt-tw-003}, one concludes that 
\begin{equation}
    \label{cnt-tw-007}
    S_c(\xi) = c e^\xi \qquad \forall \ \xi < 0,
\end{equation}
and in particular,
\begin{equation}
\label{cnt-tw-008}
S_c(0)=c,
\qquad
\partial_\xi S_c(0)=c.
\end{equation}

We now construct the active branch for \(\xi>0\) and show that $ \xi_1 = \infty $. The positivity of
\(\sigma_c\) cannot be deduced solely from continuity at the interface. Instead,
it must follow from the reduced ODE dynamics.

Define
\begin{equation}\label{def:Y-c}
Y_c:=\partial_\xi S_c.
\end{equation}
On any active interval where
$
\sigma_c>0,
$
the zero-flux relation \eqref{eq:tw-03}
reduces to
\[
-c+(1-\sigma_c)Y_c=0.
\]
or, equivalently, 
\begin{equation}
\sigma_c=1-\frac{c}{Y_c}.
\label{eq:tw-sigma-from-Y}
\end{equation}

Substituting \eqref{eq:tw-sigma-from-Y} into \eqref{eq:tw-02} yields
\[
-\partial_\xi Y_c+S_c
=
1-\frac{c}{Y_c}.
\]
Hence we have obtained the following first-order system:
\begin{subequations}\label{sys:S-Y}
\begin{align}
\partial_\xi S_c&=Y_c,\label{sys:S-Y-1} \\
\partial_\xi Y_c&=S_c-1+\dfrac{c}{Y_c}, \label{sys:S-Y-2}
\end{align}
\end{subequations}
and thanks to \eqref{cnt-tw-008}
\begin{equation}
    \label{cnt-S-Y-initial}
    S_c(0)=c > 0,
    \qquad
    Y_c(0)=c > 0.
\end{equation}
Then the Picard existence theory applies, and there is a
unique local solution to \eqref{sys:S-Y}--\eqref{cnt-S-Y-initial}. In the following lemma, we show that this solution can be extended to all $ \xi \in (0,\infty) $:

\begin{lemma}\label{lm:S-Y-cnt-global}
    The solution to system \eqref{sys:S-Y}--\eqref{cnt-S-Y-initial} exists and $ Y_c > c $ for all $ \xi > 0 $. In particular, both $ S_c $ and $ Y_c $ are monotonically increasing. Therefore thanks to \eqref{eq:tw-sigma-from-Y}, $ \sigma_c > 0 $ for all $ \xi> 0 $ and $ \xi_1 = \infty $. 
\end{lemma}


\begin{proof}

Since $ \sigma_c > 0 $ for $ 0 < \xi < \xi_1 $, we have, thanks to \eqref{eq:tw-sigma-from-Y}, 
\begin{equation}
    \label{cnt-tw-009}
    Y_c(\xi) > c > 0 \qquad \forall \ 0 < \xi < \xi_1. 
\end{equation}
From \eqref{sys:S-Y}, one then has
\begin{equation}
    \label{cnt-tw-010}
    \partial_\xi S_c = Y_c > c > 0,
\end{equation}
and in particular $ S_c $ is strictly monotonically increasing and 
\begin{equation}
\label{cnt-tw-010-1}
    S_c > c.
\end{equation}
Thus for $ \forall \ 0 <  \xi < \xi_1 $, thanks to \eqref{cnt-tw-009}, taking $ \partial_\xi $ in \eqref{sys:S-Y-2} yields
\begin{equation}
    \label{cnt-tw-012}
    \partial_\xi(\partial_\xi Y_c) + \frac{c}{Y_c^2} \partial_\xi Y_c = Y_c > 0 \qquad \text{and} \ \partial_\xi Y_c(0) = S_c(0) - 1 + \frac{c}{Y_c(0)} = c > 0.
\end{equation}
Hence 
\begin{equation}
    \label{cnt-tw-013}
    \partial_\xi Y_c > \partial_\xi Y_c(0) e^{-\int_0^\xi \frac{c}{Y_c^2(\xi')}\,d\xi'} > ce^{-\xi/c} > 0 \qquad \forall \ 0 < \xi < \xi_1,
\end{equation}
and thus $ \forall \ 0 < \xi < \xi_1 $,
\begin{gather}
    \label{cnt-tw-013}
    Y_c > Y_c(0) + \int_0^\xi c e^{-\xi'/c}\,d\xi' = c + c^2 (1-e^{-\xi/c}) > c, \\ 
    \label{cnt-tw-014}
    \sigma_c = 1- \frac{c}{Y_c}> 1 - \frac{c}{c + c^2(1-e^{-\xi/c})} = \frac{c(1-e^{-\xi/c})}{1 + c(1-e^{-\xi/c})} > 0.
\end{gather}

On the other hand, from \eqref{sys:S-Y} and \eqref{cnt-tw-013}, one has that
\begin{equation}
    \label{cnt-tw-014}
    \partial_\xi (S_c + Y_c ) = (S_c + Y_c - 1) + \frac{c}{Y_c} \leq S_c + Y_c. 
\end{equation}
Hence $ \forall \ 0 < \xi < \xi_1 $, thanks to \eqref{cnt-tw-010-1}, \eqref{cnt-tw-013}, and \eqref{cnt-tw-014},
\begin{equation}
    \label{cnt-tw-015}
    c < S_c,\ Y_c \leq S_c + Y_c \leq 2 c e^\xi.
\end{equation}
Therefore, with a continuity argument, we finish the proof of the lemma. 
\end{proof}

We now determine the asymptotic behavior of $ (S_c, Y_c) $ as \(\xi\to+\infty\). 
\begin{lemma}
    \label{lm:asymptotic-S-Y-cnt} There exists $ 0 < c_1 < 2c $, such that $ \forall \ \xi > 0 $, 
    \begin{equation}
        \label{cnt-tw-022}
        c_1 e^\xi  \leq S_c (\xi),\ Y_c(\xi) \leq 2c e^\xi,
    \end{equation}
    and thus
    \begin{equation}
        \label{cnt-tw-023}
        1-\frac{c}{c_1 e^\xi}\leq \sigma_c(\xi) \leq 1- \frac{1}{2 e^\xi}.
    \end{equation}
\end{lemma}

\begin{proof}
Since $ \xi_1 = \infty $, \eqref{cnt-tw-010} and \eqref{cnt-tw-012} hold for all $ \xi > 0 $. This implies that 
\begin{equation}
    \label{cnt-tw-016}
    S_c, Y_c \rightarrow \infty \qquad \text{as} \quad \xi \to \infty,
\end{equation}
and thanks to \eqref{eq:tw-sigma-from-Y},
\begin{equation}
    \label{cnt-tw-017}
    \sigma_c \rightarrow 1^- \qquad \text{as} \quad \xi \to \infty.
\end{equation}

Moreover, from \eqref{sys:S-Y}, one has that
\begin{gather}
\label{cnt-tw-018}
    \partial_\xi (S_c + Y_c - 1 ) = (S_c + Y_c - 1) + \frac{c}{Y_c} \geq S_c + Y_c - 1,\\
    \intertext{and}
    \label{cnt-tw-019}
    \partial_\xi(S_c- 1 - Y_c) = - (S_c- 1 - Y_c) - \frac{c}{Y_c}.
\end{gather}
Hence, one can calculate that for $ \forall \ \xi_2 \geq 0 $
\begin{gather}
\label{cnt-tw-020}
    S_c + Y_c - 1 \geq e^{\xi - \xi_2} (S_c(\xi_2) + Y_c(\xi_2) - 1 ),  \\
    \intertext{and}
\label{cnt-tw-021}
\vert S_c - 1 - Y_c \vert = \vert - e^{-\xi} - \int_0^\xi \frac{c e^{\xi'-\xi}}{Y_c(\xi')}\,d\xi' \vert \leq 2.
\end{gather}
In particular, thanks to \eqref{cnt-tw-016}, there exists $ \xi_2 $ such that the right hand side of \eqref{cnt-tw-020} is positive, and therefore, this verifies \eqref{cnt-tw-022}. Together with \eqref{eq:tw-sigma-from-Y}, this finishes the proof. 
\end{proof}

\begin{proof}[Proof of Proposition \ref{prop:tw-cnt}]
    The proof of Proposition \ref{prop:tw-cnt} follows from Lemmas \ref{lm:S-Y-cnt-global} and \ref{lm:asymptotic-S-Y-cnt}. In particular, it is easy to verify $(S_c, \sigma_c)$ is a continuous traveling wave profile for system \eqref{sys:hyperbolicKS}.
\end{proof}

\subsubsection{Right-going traveling wave with jump}
\label{subsubsec:rg-jump-tw}

We will construct the class of right-going traveling wave ($ c > 0 $) with jump based on Propositions \ref{prop:no-jump-tw} and \ref{prop:tw-cnt}. Indeed, thanks to Proposition \ref{prop:no-jump-tw} and Lemma \ref{lm:no-asymptotic-vaccum}, there are $\xi_0, \xi_1 \in \mathbb R$, $ \xi_0 < \xi_1 $,  such that
\begin{equation}
    \label{jm-tw-001}
    \sigma_c = 0 \quad \text{for} \ \forall \xi \leq \xi_0 \qquad \text{and} \qquad \sigma_c > 0 \quad \text{for} \ \forall \ \xi_0 < \xi < \xi_1,
\end{equation}
and $ \xi= \xi_0 $ is a continuous point of $ \sigma_c $. Without loss of generality, we assume that $ \xi_0 = 0 $ after translation. Then, $ (\sigma_c,S_c) $ is given by Proposition \ref{prop:tw-cnt} for $ \xi < \xi_1 $. 

Let $ \xi = \xi_1 $ be a jump. Proposition \ref{prop:no-jump-tw} implies that $ \sigma_c(\xi_1 +) = 0 $. In addition, Proposition \ref{prop:tw-cnt} yields that $ S_c(\xi_1) > c $ and $ \partial_\xi S_c(\xi_1) = Y_c(\xi_1) > c $. Let $ \xi_2 $, $ \xi_1 \leq \xi_2 \leq \infty $, be the continuous point such that 
\begin{equation}
    \label{jm-tw-002}
    \sigma_c = 0 \quad \forall \ \xi_1 < \xi < \xi_2, \quad \text{and} \quad \sigma_c(\xi) > 0 \quad \xi \in (\xi_2,\xi_2 + \delta) \quad \text{for some} \ \delta > 0.  
\end{equation}
In particular, \eqref{eq:tw-03} implies
\begin{equation}
    \label{jm-tw-002-2}
    -c + (1-\sigma_c) \partial_\xi S_c = 0 \quad \forall \ \xi_2 < \xi < \xi_2 + \delta.
\end{equation}
Then following argument as in Proposition \ref{prop:tw-cnt}, one can derive that, for $ \xi_1 < \xi < \xi_2 $
\begin{subequations}
    \label{jm-tw-003}
    \begin{align}
        \partial_\xi S_c & = Y_c >0, \\
        \partial_\xi Y_c & = S_c >0.
    \end{align}
\end{subequations}
Therefore, one can conclude that
\begin{equation}
    \label{jm-tw-004}
    S_c(\xi_2) > c \quad \text{and} \quad \partial_\xi S_c(\xi_2) = Y_c(\xi_2) > c.
\end{equation}
This leads to a contradiction since, thanks to \eqref{jm-tw-002-2},
\begin{equation}
    \label{jm-tw-004-1}
    0 = \sigma_c(\xi_2+) = \lim_{\xi \to \xi_2+} \lbrace 1-\frac{c}{\partial_\xi S_c}\rbrace > 0. 
\end{equation}
Therefore, $ \xi_2 = \infty $ and 
\begin{equation}
    \label{jm-tw-005}
    \sigma_c = 0 \qquad \text{for all} \ \forall \ \xi > \xi_1,
\end{equation}
i.e., $ \xi_1 $ is the only possible jump. In particular, solving \eqref{jm-tw-003} with \eqref{jm-tw-004} leads to 
the same asymptotic behavior of $ S_c, Y_c $ as in \eqref{cnt-tw-022}. 

We have shown the following proposition:
\begin{proposition}[Right-going discontinuous traveling wave]\label{prop:rg-jump-tw} Let $ c > 0 $. For arbitrary $ \xi_0, \xi_1 \in \mathbb R $ and $ \xi_0 < \xi_1 $. There exists a traveling wave solution $ (\sigma_c, S_c) $ to system \eqref{sys:tw} such that 
\begin{equation}
    \label{tw-discontinuous}
    \sigma_c = 0 \qquad \text{for} \quad \xi \leq \xi_0 \quad \text{and} \quad \xi > \xi_1,
\end{equation}
and $ \sigma_c $ is continuous and monotonically increasing for $ \xi \in [\xi_0,\xi_1) $. In addition, $ \xi = \xi_0 $ is a continuous point and $ \xi = \xi_1 $ is the only jump for $ \sigma_c $. The profile for $(\sigma_c,S_c)$ is solved by \eqref{sys:S-Y} for $ \xi \in (\xi_0,\xi_1) $ and \eqref{jm-tw-003} for $ \xi \in (\xi_1,\infty) $. 
\end{proposition}

\begin{figure}[htbp]
\centering
\begin{tikzpicture}
\begin{axis} [
    xlabel={$\xi = x - ct$},
    ylabel={$\sigma_c(\xi)$},
    xmin=-2, xmax=5,
    ymin=-0.2, ymax=1.2,
    samples=200,
    grid=both,
    grid style={dashed, gray!30},
    width=12cm,
    height=4cm,
    legend pos=north west
]

\addplot[domain=-2:0, blue, ultra thick] {0};
\addplot[domain=0:3, blue, ultra thick] { 1 - 1 / (exp(x))};
\addplot[domain=3:6, blue, ultra thick] {0};

\addplot [blue, ultra thick, forget plot] coordinates {
    (3, {1 - exp(-3)}) 
    (3, 0) 
};

\addlegendentry{$\sigma_c(\xi)$}



\end{axis}
\end{tikzpicture}
\caption{Sample of traveling wave profile $\sigma_c(\xi), \ c > 0 $, from Proposition \ref{prop:rg-jump-tw}}
\label{fig:sigma_c_traveling_rg-jump}
\end{figure}
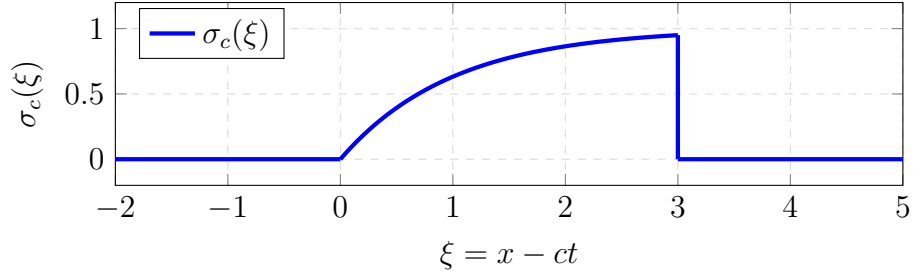

\subsection{Left-going traveling wave solutions}
\label{subsec:left-going-tw}

\subsubsection{Left-going continuous traveling wave, i.e., $ c < 0 $}
\label{subsubsec:lg-continous-tw}

We consider $ c < 0 $. Following the same arguments as in section \ref{subsubsec:rg-continous-tw}, without loss of generality, assume that $ \xi_0 = 0 $ is the point such that
\begin{equation}
    \label{cnt-tw-l:001}
    \sigma_c = 0 \quad \text{for} \ \xi \leq 0 \qquad \text{and} \ \sigma_c > 0 \quad \text{for} \ \xi \in (0,\xi_1) \quad \text{for some} \ \xi_1 > 0. 
\end{equation}
Then one has that 
\begin{gather}
    \label{cnt-tw-l:002}
    S_c(\xi) = c e^\xi \qquad \text{for} \ \xi \leq 0, \\
    \intertext{and for $ \xi \in (0,\xi_1) $,}
    \label{cnt-tw-l:003}
    \begin{cases}
        \partial_\xi S_c = Y_c, \\
        \partial_\xi Y_c = S_c - 1 + \frac{c}{Y_c},\\
        \sigma_c = 1- \frac{c}{Y_c},
    \end{cases}\\
    \intertext{with}
    \label{cnt-tw-l:004}
    Y_c(0) = S_c(0) = c < 0,\\
    \intertext{and since $ \sigma_c > 0 $, }
    \label{cnt-tw-l:005}
    Y_c < c<0 \qquad \text{for} \ \xi \in (0,\xi_1).
\end{gather}
Therefore, one has that for $ \xi \in (0,\xi_1) $, 
\begin{gather}
    \label{cnt-tw-l:006}
    \partial_\xi S_c < c < 0, \qquad \text{which imples that} \ S_c < S_c(0) = c <0, \\
    \intertext{and}
    \label{cnt-tw-l:007}
    \partial_\xi Y_c^2 = \underbrace{(S_c -1 )Y_c}_{> 0} + c \geq (c-1)c + c = c^2 > 0.
\end{gather}
In particular, \eqref{cnt-tw-l:005} and \eqref{cnt-tw-l:007} imply that for $ \xi \in (0,\xi_1) $, 
\begin{equation}
    \label{cnt-tw-l:008}
    Y_c < c \sqrt{1+\xi^2} < c < 0,
\end{equation}
which implies that 
\begin{equation}
    \label{cnt-tw-l:009}
    1 > \sigma_c(\xi) = 1-  \underbrace{\frac{c}{Y_c}}_{>0} > 1 - \frac{1}{\sqrt{1+\xi^2}} > 0.
\end{equation}
Similarly as before, one also has
\begin{gather}
    \label{cnt-tw-l:010}
    S_c + Y_c - 1< \partial_\xi(S_c + Y_c) = S_c + Y_c - 1 + \frac{c}{Y_c} = S_c + Y_c - \sigma_c < S_c + Y_c,\\
    \intertext{and}
    \label{cnt-tw-l:011}
    - (S_c - Y_c) <\partial_\xi(S_c - Y_c) = - (S_c - Y_c) + \sigma_c < - (S_c - Y_c) + 1.
\end{gather}
Solving \eqref{cnt-tw-l:010} and \eqref{cnt-tw-l:011} implies that, for some $ 0  < -2c < c_1 $ and for all $ \xi \in (0,\xi_1) $,
\begin{equation}
    \label{cnt-tw-l:012}
    -c_1 e^\xi < S_c,\ Y_c < 2c e^\xi < 0, \qquad \text{and} \qquad  1 - \frac{1}{2 e^\xi} < \sigma_c < 1 + \frac{c}{c_1 e^\xi}.
\end{equation}
Therefore, by a continuity argument, $ \xi_1 = \infty $ and we have proved the following proposition: 
\begin{proposition}[Left-going continuous traveling wave]
    \label{prop:lg-cnt-tw}
    Let \(c<0\). Then there exists a continuous global traveling-wave profile of \eqref{sys:hyperbolicKS}, satisfying \eqref{sys:tw}--\eqref{eq:tw-left-vacuum}.
After translation, the profile is given on the left by
\[
\sigma_c(\xi)=0,
\qquad
S_c(\xi)=ce^\xi,
\qquad
\xi\le0.
\]
For \(\xi>0\), the pair \((S_c,Y_c :=\partial_\xi S_c )\)
solves the initial value problem
\begin{equation}
\begin{cases}
\partial_\xi S_c=Y_c,\\[4pt]
\partial_\xi Y_c=S_c-1+\dfrac{c}{Y_c},
\end{cases}
\qquad
S_c(0)=c,
\qquad
Y_c(0)=c,
\label{eq:tw-l-system}
\end{equation}
and
\begin{equation}
\sigma_c(\xi)=1-\frac{c}{Y_c(\xi)}.
\label{eq:tw-l-sigma}
\end{equation}
Moreover,
\[
Y_c(\xi)<c<0
\qquad
\text{for all }\xi>0,
\]
and consequently,
\[
0<\sigma_c(\xi)<1
\qquad
\text{for all }\xi>0.
\]

Additionally,
\[
S_c(\xi)\to -\infty,
\qquad
Y_c(\xi)\to -\infty,
\qquad
\sigma_c(\xi)\to1^-
\qquad
\text{as }\xi\to+\infty.
\]
Thus the profile is a one-sided traveling semi-wave.
\end{proposition}

\subsubsection{Left-going traveling wave with jump}
\label{subsub:lg-jump-tw}
Thanks to Proposition \ref{prop:no-jump-tw}, the same type of discontinuous traveling wave as in Proposition \ref{prop:rg-jump-tw}, i.e., a jump after a continuous non-trivial profile, is impossible. However, another type of discontinuous traveling wave emerges, i.e., a jump {\bf before} a continuous non-trivial profile. Without loss of generality, let $ \xi_0 = 0 $ be the jump, and 
\begin{equation}
    \label{jump-tw-l:001}
    \sigma_c(\xi) = 0, \qquad S_c(\xi) = Ae^\xi,
\end{equation}
for $ \xi< 0 $.
Then the jump condition in Proposition \ref{prop:no-jump-tw} implies that
\begin{equation}
    \label{jump-tw-l:002}
    \sigma_c(0+) = 1 - \frac{c}{A} \in (0,1],
\end{equation}
which requires that 
\begin{equation}
    \label{jump-tw-l:003}
    A < c < 0.
\end{equation}
Then for $ \xi >0 $, the same arguments as in \eqref{cnt-tw-l:003}--\eqref{cnt-tw-l:012} apply with \eqref{cnt-tw-l:004} replaced by
\begin{equation}
    \label{jump-tw-l:004}
    Y_c(0) = S_c(0) = A < c < 0. 
\end{equation}

We have proved the following proposition:
\begin{proposition}[Left-going traveling wave with jump] 
    \label{prop:lg-jump-tw}
    Let $ c <0 $. For arbitrary $ A < c $, upon translation, there exists a global traveling wave profile with jump at $ \xi = 0 $ such that
    \begin{equation}
        \sigma_c = 0 \qquad \text{for} \ \xi <0,
    \end{equation}
    and $ (\sigma_c, S_c) $, $ \xi \geq 0 $, is the solution to \eqref{cnt-tw-l:003} with initial data \eqref{jump-tw-l:004}. The asymptotics as $ \xi \rightarrow \infty $ of $ (\sigma_c, S_c) $ is the same as the continuous traveling wave constructed in Proposition \ref{prop:lg-cnt-tw}.
\end{proposition}

\begin{figure}[htbp]
\centering
\begin{tikzpicture}
\begin{axis} [
    xlabel={$\xi = x - ct$},
    ylabel={$\sigma_c(\xi)$},
    xmin=-2, xmax=5,
    ymin=-0.2, ymax=1.2,
    samples=200,
    grid=both,
    grid style={dashed, gray!30},
    width=12cm,
    height=4cm,
    legend pos=north west
]

\addplot[domain=-2:0, blue, ultra thick] {0};
\addplot[domain=0:6, blue, ultra thick] { 1 - 3 / (4 * exp(x))};

\addplot [blue, ultra thick, forget plot] coordinates {
    (0, {1/4}) 
    (0, 0) 
};

\addlegendentry{$\sigma_c(\xi)$}



\end{axis}
\end{tikzpicture}
\caption{Sample of traveling wave profile $\sigma_c(\xi), \ c < 0 $, from Proposition \ref{prop:lg-jump-tw}}
\label{fig:sigma_c_traveling_rg-jump}
\end{figure}
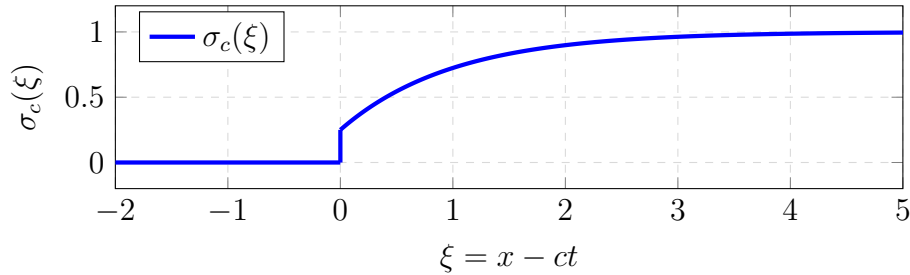

\section{Non-vacuum far field}
\label{sec:non-vacuum-far-field}

In this section, we consider system \eqref{sys:hyperbolicKS} with non-vacuum far field \eqref{non-vacuum-far-field}.

\subsection{The case when $ \sigma_\infty = 1 $}
\label{subsec:saturation-far-field}

For $ \sigma_\infty = 1 $, 
one considers $ \eta:= \sigma_\infty - \sigma = 1 - \sigma $ and $ Q:= \sigma_\infty - S = 1 - S $. Then $ (\eta,  Q) $ satisfies
\begin{subequations}
    \begin{align}
        \dt \eta + \dx(\eta(1-\eta) \dx Q) & = 0, \\
        - \partial_{xx} Q + Q  & = \eta,\\
        Q(-\infty) = \eta(-\infty) & = 0.
    \end{align}
\end{subequations}
That is, $ (\eta, Q) $ satisfies the same system of equations and asymptotics as in \eqref{sys:hyperbolicKS}--\eqref{far-field-vacuum}. Therefore, the construction of stationary and traveling wave solutions follow from sections \ref{sec:stationary}--\ref{sec:traveling_wave}, and therefore is omitted here. 

Without loss of generality, we only consider 
\begin{equation}
    \label{eq:nv-far-field-000}
    \sigma_\infty \in (0,1).
\end{equation}

\subsection{Stationary solutions}
\label{subsec:station-non-vacuum-ff}
Recalling \eqref{eq:stn-08}--\eqref{eq:stn-10}, with $ \sigma_\infty \in (0,1) $, the stationary solution can only start with Type
III at $ - \infty $.

In this case, one has that 
\begin{equation}
    \sigma_s(x) = S_s(x) = \sigma_{\infty} \in (0,1), \qquad x \leq x_0. 
\end{equation}

{\noindent\bf Case 1: Connecting to Type I at $ x = x_0 $.}
In this case, thanks to \eqref{eq:jump-1}, we have
\begin{equation}
    \label{eq:stn-nv-101}
    S_s(x) = \begin{cases}
        \sigma_{\infty}, & x \leq x_0, \\
        A_0 e^x + B_0 e^{-x}, & x_0 < x \leq x_1, 
    \end{cases} \qquad \sigma_s(x) = \begin{cases}
        \sigma_{\infty} > 0, & x \leq x_0, \\
        0, & x_0 < x \leq x_1, 
    \end{cases}
\end{equation}
with 
\begin{equation}
    \label{eq:stn-nv-102}
    S_s(x_0) = \sigma_\infty = A_0 e^{x_0} + B_0 e^{-x_0} \in (0,1), \qquad 0 = \dx S_s(x_0) = A_0 e^{x_0} - B_0 e^{-x_0}. 
\end{equation}
In particular, one has that 
\begin{equation}
    \label{eq:stn-nv-103}
    A_0 e^{x_0} = B_0 e^{-x_0} = \frac{1}{2} \sigma_\infty, \qquad A_0, \ B_0 > 0.  
\end{equation}
Thus, at $ x = x_1 > x_0 $, one will have
\begin{equation}
    \label{eq:stn-nv-104}
    \dx S_s(x_1) = A_0 e^{x_1} - B_0 e^{-x_1} >  A_0 e^{x_0} - B_0 e^{-x_0} = 0. 
\end{equation}
However, since $ \sigma_s(x_1^-) = 0 < \sigma_s(x_1^+) $, \eqref{eq:stn-nv-104} violates the entropy condition \eqref{eq:jump-3}. Hence $ x_1 = \infty $, $ x_0 $ is the only jump, and the stationary solution is given by \eqref{eq:stn-nv-105}--\eqref{eq:stn-nv-106}.

\begin{figure}[htbp]
\centering
\begin{tikzpicture}
\begin{axis}[
    xlabel={$x$},
    ylabel={Value},
    xmin=-4, xmax=4,
    ymin=-0.25, ymax=1.25,
    samples=200,
    grid=both,
    legend pos=south west,
    width=13cm,
    height=5cm
]
\addplot[domain=-5:0, blue, ultra thick] {1/4};
\addplot[domain=0:2, blue, ultra thick, forget plot] { 1/8 *exp(x) + 1/8 *exp(-x)};
\addlegendentry{$S_s(x)$}

\addplot[domain=-5:0, red, dashed, ultra thick] {1/4};
\addplot[domain=0:5, red, dashed, ultra thick, forget plot] {0};
\addlegendentry{$\sigma_s(x)$}

\draw[red, dotted, ultra thick] (axis cs:0,0) -- (axis cs:0,1/4);

\end{axis}
\end{tikzpicture}
\caption{Solution \eqref{eq:stn-nv-105}--\eqref{eq:stn-nv-106} with $ x_0 = 0, \sigma_\infty = 1/4 $}
\end{figure}
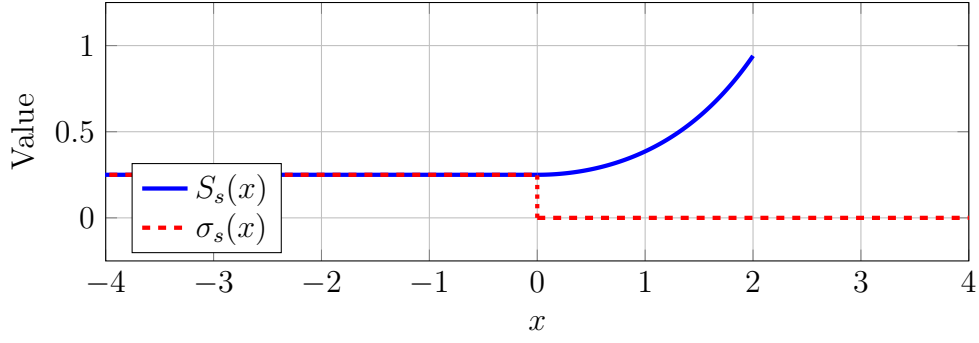

{\noindent\bf Case 2: Connecting to Type II at $ x = x_0 $.} In this case, thanks to \eqref{eq:jump-1}, we have, without loss of generality, $ \sigma_\infty \in (0,1) $, and 
\begin{equation}
    \label{eq:stn-nv-201}
    S_s(x) = \begin{cases}
        \sigma_{\infty}, & x \leq x_0, \\
        1 + C_0 e^x + D_0 e^{-x}, & x_0 < x \leq x_1, 
    \end{cases} \qquad \sigma_s(x) = \begin{cases}
        \sigma_{\infty} > 0, & x \leq x_0, \\
        1, & x_0 < x \leq x_1, 
    \end{cases}
\end{equation}
with 
\begin{equation}
    \label{eq:stn-nv-202}
    S_s(x_0) = \sigma_\infty = 1 + C_0 e^{x_0} + D_0 e^{-x_0} \in (0,1), \qquad 0 = \dx S_s(x_0) = C_0 e^{x_0} - D_0 e^{-x_0}. 
\end{equation}
In particular, 
\begin{equation}
    \label{eq:stn-nv-203}
    C_0e^{x_0} = D_0 e^{-x_0} = \frac{\sigma_\infty - 1}{2} < 0, \qquad C_0, \ D_0 < 0. 
\end{equation}
Thus, at $ x= x_1 > x_0 $, one will have
\begin{equation}
    \label{eq:stn-nv-204}
    \dx S_s(x_1) = C_0 e^{x_1} - D_0 e^{-x_1} < C_0 e^{x_0} - D_0 e^{-x_0} =0. 
\end{equation}
However, since $ \sigma_s(x_1^-) = 1 > \sigma_s(x_1^+) $, \eqref{eq:stn-nv-204} violates the entropy condition \eqref{eq:jump-3}. Hence $ x_1 = \infty $, $ x_0 $ is the only jump, and the stationary solution is given by \eqref{eq:stn-nv-205}--\eqref{eq:stn-nv-206}.
\begin{figure}[htbp]
\centering
\begin{tikzpicture}
\begin{axis}[
    xlabel={$x$},
    ylabel={Value},
    xmin=-4, xmax=4,
    ymin=-1.25, ymax=1.25,
    samples=200,
    grid=both,
    legend pos=south west,
    width=13cm,
    height=5cm
]
\addplot[domain=-5:0, blue, ultra thick] {1/4};
\addplot[domain=0:2, blue, ultra thick, forget plot] { 1 - 3/8 *exp(x) - 3/8 *exp(-x)};
\addlegendentry{$S_s(x)$}

\addplot[domain=-5:0, red, dashed, ultra thick] {1/4};
\addplot[domain=0:5, red, dashed, ultra thick, forget plot] {1};
\addlegendentry{$\sigma_s(x)$}

\draw[red, dotted, ultra thick] (axis cs:0,1/4) -- (axis cs:0,1);

\end{axis}
\end{tikzpicture}
\caption{Solution \eqref{eq:stn-nv-205}--\eqref{eq:stn-nv-206} with $ x_0 = 0, \sigma_\infty = 1/4 $}
\end{figure}
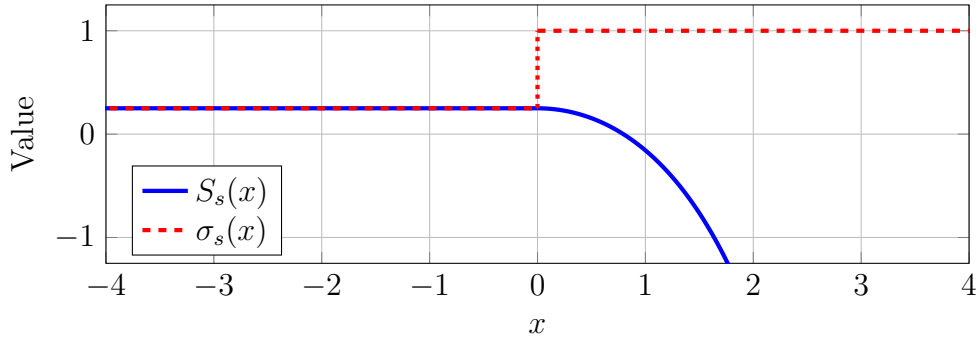

\subsection{Traveling wave solutions}
\label{subsec:tw-non-vacuum-ff}

Consider \eqref{eq:tw-ansatz}--\eqref{sys:tw} with 
\begin{equation}
    \label{eq:tw-nv-ff}
    \sigma_c(\xi), \ S_c(\xi) \rightarrow \sigma_\infty, \qquad \text{as} \ \xi \rightarrow - \infty. 
\end{equation}
Then integrating \eqref{eq:tw-01} once yields
\begin{equation}
    \label{eq:tw-nv-001}
    - c \sigma_c + \sigma_c (1-\sigma_c) \partial_\xi S_c = - c \sigma_\infty.
\end{equation}
In particular, since $ \sigma_\infty \in (0,1) $, the dichotomy as in the previous cases does not exist. Moreover, it is easy to show that
\begin{lemma}
\label{lm:no-trivial-point}
    For $ \sigma_\infty \in (0,1) $ and $ c \neq 0 $, any solution to \eqref{eq:tw-nv-001} satisfies
    \begin{equation}
        \label{eq:tw-nv-002}
        \sigma_c(\xi^\pm) \neq 1 \qquad \text{and} \qquad \sigma_c(\xi^\pm) \neq 0 \qquad \forall \ \xi \in \mathbb R. 
    \end{equation}
\end{lemma}
\begin{proof}
    Without loss of generality, suppose that $ \sigma_c (\xi_0+) = 0 $. Then from \eqref{eq:tw-nv-001}, one has that as $ \xi \rightarrow \xi_0^+ $, 
    \begin{equation}
        0 = - c \sigma_\infty \neq 0,
    \end{equation}
    which reaches contradiction. 

    On the other hand, if $ \sigma_c(\xi_1+) = 1 $, then from \eqref{eq:tw-nv-001}, one has that as $ \xi \rightarrow \xi_1^+ $,
    \begin{equation}
        - c = - c \sigma_\infty,
    \end{equation}
    which contradicts the fact $ \sigma_\infty \in (0,1)$. 
\end{proof}

Thanks to lemma \ref{lm:no-trivial-point}, one can easily conclude the following lemma from \eqref{eq:jump-1} and \eqref{eq:tw-nv-001}:

\begin{lemma}
\label{lm:cnt-ttrivial-point}
    Any point $ \xi_0 \in \mathbb R $ satisfying $ \sigma_c(\xi_0^\pm) = \sigma_\infty $ is always a continuous point for $ \sigma_c $ and $ S_c $, and thus $ \sigma_c(\xi_0) = \sigma_\infty $. 
\end{lemma}
\begin{proof}
    Thanks to \eqref{eq:jump-1}, $ \partial_\xi S_c $ is continuous at $ \xi_0 $. Thus $ \sigma_c(\xi_0^\pm) = \sigma_\infty $ and \eqref{eq:tw-nv-001}, together with lemma \ref{lm:no-trivial-point}, imply that $ \partial_\xi S_c(\xi_0) = 0 $. Therefore \eqref{eq:tw-nv-001} implies that $ \sigma_c (\xi_0) = \sigma_\infty $, i.e., $\xi_0 $ is a continuous point of $ \sigma_c $. It is trivial that $ S_c $ is continuous at $ \xi_0 $ thanks to \eqref{eq:jump-1}. 
\end{proof}

\begin{lemma}
\label{lm:no-ttrivial-point}
    If there exists $ \xi_0 \in \mathbb R $ such that $ S_c(\xi_0) = \sigma_c(\xi_0) = \sigma_\infty $, then 
    \begin{equation}
        \label{eq:tw-nv-003}
        \sigma_c = S_c = \sigma_\infty \qquad \forall \ \xi \in \mathbb R. 
    \end{equation}
    We call \eqref{eq:tw-nv-003} the trivial solution to \eqref{sys:tw} with \eqref{eq:tw-nv-ff}.
\end{lemma}
\begin{proof}
    From \eqref{eq:tw-nv-001}, 
    one immediately has that
    \begin{equation}
        \label{eq:tw-nv-004}
        \partial_\xi S_c(\xi_0) = 0 . 
    \end{equation}
    Let $ R_c := S_c - \sigma_\infty $, $ T_c := \partial_\xi R_c = \partial_{\xi}S_c $, and $ \eta_c := \sigma_c - \sigma_\infty $. Then from \eqref{eq:tw-nv-001} and \eqref{eq:tw-02} imply 
    \begin{equation}
    \label{eq:tw-nv-005}
        \partial_\xi R_c = T_c  = \frac{c \eta_c}{\sigma_c(1-\sigma_c)}, \qquad \partial_\xi T_c 
        = R_c - \frac{\sigma_c (1-\sigma_c)}{c}T_c, 
    \end{equation}
    with $ (R_c, T_c) \vert_{\xi = \xi_0} = 0 $. 
    Moreover, one can verify that, after applying the Cauchy inequality
    \begin{equation}
    \label{eq:tw-nv-005-2}
        \frac{1}{C}(R_c^2 + T_c^2) \leq \partial_\xi (R_c^2 + T_c^2) = 4 T_c R_c - \frac{2\sigma_c (1-\sigma_c)}{c}T_c^2 \leq C (R_c^2 + T_c^2), 
    \end{equation}
    for some constant $ C \in \mathbb R $, since $ \sigma_c (1-\sigma_c) \in \lbrack 0,\frac{1}{2}\rbrack $. 
    This implies $ R_c \equiv T_c \equiv 0 $, which yields \eqref{eq:tw-nv-003}.
\end{proof}


\begin{lemma}
    \label{lm:no-crossing-asymptotics}
    Other than the trivial solution \eqref{eq:tw-nv-003}, there is no solution $ (\sigma_c, S_c) $ to \eqref{sys:tw} with \eqref{eq:tw-nv-ff} admitting a point $ \sigma_c (\xi_0) = \sigma_\infty $ for some $ \xi_0 \in \mathbb R $.

    In other words, 
    if there exists a point $ \xi_0 \in \mathbb R $ such that $ \sigma_c(\xi_0) = \sigma_\infty $, then $ (\sigma_c,S_c) $ is the trivial solution \eqref{eq:tw-nv-003}. 
\end{lemma}

\begin{proof}
    We prove by contradiction. Suppose that the set 
    \begin{equation}
        \label{001-tw-nv}
        \mathsf M: = \lbrace \xi \vert \sigma_c(\xi) = \sigma_\infty \rbrace \neq \emptyset. 
    \end{equation}
    Let 
    \begin{equation}
        \label{002-tw-nv}
        \xi_0 = \inf \mathsf M. 
    \end{equation}
    
    First, we claim that $ \xi_0 > -\infty $. Otherwise, there exists $ \xi_1 \in \mathbb R $ such that for all $ \xi < \xi_1 $,  $ \sigma_c (\xi) = \sigma_\infty $ and $ S_c (\xi) = \sigma_\infty $ thanks to \eqref{eq:tw-nv-001} and \eqref{eq:tw-nv-ff}. Lemma \ref{lm:no-ttrivial-point} implies that $ (\sigma_c,S_c) $ is a trivial solution. 
    
    Next, without loss of generality, assume that $ \sigma_c > \sigma_\infty $ for $ \xi < \xi_0 $ and $ c > 0 $. Then from \eqref{eq:tw-nv-001} (or equivalently \eqref{eq:tw-nv-005}), one has that
    \begin{equation}
        \label{003-tw-nv}
        \partial_\xi S_c = \frac{c (\sigma_c - \sigma_\infty)}{\sigma_c (1-\sigma_c)} > 0 \quad \forall \ \xi < \xi_0 \qquad \text{and} \qquad  \partial_\xi S_c(\xi_0) = 0.
    \end{equation}
    Hence $ S_c(\xi) > S_c(-\infty) = \sigma_\infty $ for all $ \xi \leq  \xi_0 $. However, from \eqref{eq:tw-02}, one has that
    \begin{equation}
    \label{004-tw-nv}
        \partial_{\xi\xi} S_c (\xi_0) = S_c(\xi_0) - \sigma_c(\xi_0) > 0.
    \end{equation}
    Thanks to \eqref{eq:jump-1} and \eqref{eq:tw-nv-001}, $ S_c $ and $ \sigma_c $ are continuous at $ \xi_0 $. Thus for some $ \xi_\delta < \xi_0 $, one can calculate that
    \begin{equation}
         0 < \frac{c (\sigma_c - \sigma_\infty)}{\sigma_c (1-\sigma_c)}\Big\vert_{\xi = \xi_1} = \partial_\xi S_c(\xi_1) = - \int_{\xi_1}^{\xi_0} \partial_{\xi\xi} S_c(\xi') \, d\xi' < 0,
    \end{equation}
    which is impossible.
    The case when $ c > 0 $ and/or $ \sigma_c(\xi) < \sigma_\infty $ for $\xi< \xi_0 $ follows with similar arguments. 
    This finishes the proof of the lemma. 
\end{proof}

\begin{lemma}
    \label{lm:no-jump-tw-nv}
    There is no jump in any non-trivial traveling wave solution to system \eqref{sys:tw} with \eqref{eq:tw-nv-ff}.

    In other words, all traveling solutions to system \eqref{sys:tw} with \eqref{eq:tw-nv-ff} are continuous.
\end{lemma}

\begin{proof}
Let $ \xi_0 $ be a potential jump for $ \sigma_c $.
Without loss of generality, we consider the case when $ c > 0 $ and $ \sigma_c(\xi_0^-) > \sigma_\infty $. Then thanks to \eqref{eq:jump-1} and \eqref{eq:tw-nv-001},
\begin{equation}
    \label{005-tw-nv}
    0< \partial_\xi S_c(\xi_0) = \frac{c(\sigma_c - \sigma_\infty)}{\sigma_c (1-\sigma_c)}\Big\vert_{\xi \rightarrow \xi_0^-} = \frac{c(\sigma_c - \sigma_\infty)}{\sigma_c (1-\sigma_c)}\Big\vert_{\xi \rightarrow \xi_0^+}.
\end{equation}
In particular, $ \sigma_c(\xi_0^-) $ and $ \sigma_c(\xi_0^+) $ are the two solutions to the algebraic equation
\begin{equation}
    \label{006-tw-nv}
    \partial_\xi S_c(\xi_0) = \frac{c(x-\sigma_\infty)}{x(1-x)} \quad \text{i.e.} \quad p(x):=\partial_\xi S_c(\xi_0) x^2 + (c-\partial_\xi S_c(\xi_0)) x - c\sigma_\infty = 0. 
\end{equation}
However, $ p $ is a quadratic function with a positive second order term, and since $ \sigma_\infty \in (0,1) $, 
\begin{equation}
    \label{007-tw-nv}
    p(\sigma_\infty) = \partial_\xi S_c(\xi_0) (\sigma_\infty - 1) \sigma_\infty < 0. 
\end{equation}
Thus the two solutions $ x_{1,2} $ to \eqref{006-tw-nv} satisfy $ x_1 < \sigma_\infty < x_2 $. Since $ \sigma_c(\xi_0^-) > \sigma_\infty $, this implies $ \sigma_c(\xi_0^+) < \sigma_\infty $. This leads to contradiction to \eqref{005-tw-nv} since 
\begin{equation}
    \partial_\xi S_c(\xi_0) = \frac{c(\sigma_c - \sigma_\infty)}{\sigma_c (1-\sigma_c)}\Big\vert_{\xi \rightarrow \xi_0^+} < 0.
\end{equation}

The case when $ c < 0 $ and/or $ \sigma_c(\xi_0^-) < \sigma_\infty $ follows with similar arguments. This finishes the proof of the lemma. 
\end{proof}

It follows from lemmas \ref{lm:no-crossing-asymptotics} and \ref{lm:no-jump-tw-nv} that
\begin{proposition}
    \label{prop:tw-nv-cnt}
    All traveling wave solutions $(\sigma_c, S_c)$ to system \eqref{sys:tw} with \eqref{eq:tw-nv-ff} are continuous, and satisfy 
    \begin{itemize}
        \item either $ \sigma_c > \sigma_\infty, \  c \partial_\xi S_c > 0 $ for all $ \xi \in \mathbb R $;
        \item or $ \sigma_c < \sigma_\infty, \  c \partial_\xi S_c < 0 $ for all $ \xi \in \mathbb R $;
        \item or $ (\sigma_c, S_c) = (\sigma_\infty, \sigma_\infty) $ for all $ \xi \in \mathbb R $. 
    \end{itemize}
\end{proposition}

With proposition \ref{prop:tw-nv-cnt}, we will be able to search for the non-trivial traveling wave solutions.

\subsubsection{The case when $ \sigma_c > \sigma_\infty $ and $ c > 0 $.}
\label{subsubsec:++}
Thanks to proposition \ref{prop:tw-nv-cnt}, we have $ \partial_\xi S_c > 0 $. 
Recalling \eqref{eq:tw-nv-005}--\eqref{eq:tw-nv-005-2}, 
\begin{equation}
    \label{008-tw-nv}
    \begin{cases}
        \partial_\xi R_c = T_c, \\
        \partial_\xi T_c = R_c - \frac{\sigma_c(T_c) (1-\sigma_c(T_c))}{c} T_c = R_c + \sigma_\infty - \sigma_c,
    \end{cases} \qquad \lim_{\xi\rightarrow -\infty} (R_c, T_c) = 0,
\end{equation}
where, 
    $ R_c := S_c - \sigma_\infty, \ T_c := \partial_\xi S_c > 0 $ and thanks to \eqref{008-tw-nv}
    \begin{equation}
        \label{008-tw-nv-1}
        R_c(\xi) = \int_{-\infty}^\xi T_c(\xi') \, d\xi' > 0,
    \end{equation}
    provided the solution to \eqref{008-tw-nv} exists.

Recalling $ p = p(x) $ from \eqref{006-tw-nv}, one can easily verify that
\begin{equation}
\label{009-tw-nv}
    p(0) = - c \sigma_\infty < 0, \quad p(\sigma_\infty) < 0, \quad p(1) = c(1-\sigma_\infty) > 0.
\end{equation}
Since $ p $ has a positive leading order term, one can conclude that
\begin{equation}
\label{010-tw-nv}
    \sigma_c = \sigma_c (T_c) := \frac{(T_c - c) + \sqrt{(c-T_c)^2 + 4 c T_c \sigma_\infty}}{2 T_c} = \frac{2 c \sigma_\infty}{(c - T_c) + \sqrt{(c-T_c)^2 + 4 c T_c \sigma_\infty}}.
\end{equation}
In particular, 
\begin{equation}
    \label{011-tw-nv}
    0< \sigma_\infty < \sigma_c < 1 \quad \text{and hence} \quad 0 < \frac{\sigma_c(T_c) (1-\sigma_c(T_c))}{c} \leq \frac{1}{4c}.
\end{equation}
Therefore, from \eqref{008-tw-nv}, one has that 
\begin{equation}
    \label{012-tw-nv}
    R_c + \frac{1}{4c}T_c \leq \partial_\xi (T_c + \frac{1}{2c}R_c) = R_c + (\frac{1}{2c} - \frac{\sigma_c(T_c) (1-\sigma_c(T_c))}{c})T_c < R_c + \frac{1}{2c} T_c.
\end{equation}
In particular, since $ T_c, R_c >0 $, one has that
\begin{equation}
    \label{013-tw-nv}
    (\frac{1}{4c} + 2 c)(T_c + \frac{1}{2c}R_c) \leq \partial_\xi (T_c + \frac{1}{2c}R_c) < (\frac{1}{2c} + 2c ) (T_c + \frac{1}{2c}R_c).
\end{equation}
Solving \eqref{013-tw-nv} leads to 
\begin{equation}
    \label{014-tw-nv}
    (T_c(0) + \frac{1}{2c}R_c(0)) e^{(\frac{1}{4c} + 2c)\xi} \leq T_c(\xi) + \frac{1}{2c}R_c(\xi) \leq (T_c(0) + \frac{1}{2c}R_c(0)) e^{(\frac{1}{2c} + 2c)\xi},
\end{equation}
provided that the solution to \eqref{008-tw-nv}.

\smallskip

To show the existence of non-trivial solution to \eqref{008-tw-nv}, we investigate the integral curve of the initial value problem. To be more precise, from \eqref{008-tw-nv}, one considers that
\begin{equation}
    \label{101-tw-nv}
    \frac{d Y_{++}}{d R_{++}} = \frac{\partial_\xi T_{++}^2}{\partial_\xi R_{++}} = 2 \lbrack R_{++} + \sigma_\infty - \sigma_c(\sqrt{Y_{++}}) )\rbrack, 
\end{equation}
with $ Y_{++}: = T_{++}^2 $ and $ \sigma_c = \sigma_c(T_{++}) $ given by \eqref{010-tw-nv}. Our goal is to find the integral curve of \eqref{101-tw-nv} passing through $ (R_{++}, Y_{++}) = (0,0) $. 

Meanwhile, from \eqref{010-tw-nv}, one has that $ \sigma_c (\sqrt Y_{++}) $ is smooth for $ Y_{++} > 0 $ and H\"older continuous at $ Y_{++} = 0 $, and $ \sigma_c (0) = \sigma_\infty $. Moreover, one can calculate that
\begin{equation}
    \label{102-tw-nv}
    \frac{d\sigma_c(T_{++})}{d T_{++}} = \sigma_c^2(T_{++}) \frac{\sqrt{(c-T_{++})^2 + 4c T_{++} \sigma_\infty} - \lbrack (T_{++}-c) + 2c\sigma_\infty\rbrack}{2 c \sigma_\infty \sqrt{(c-T_{++})^2 + 4c T_{++} \sigma_\infty}} > 0, \qquad \forall \ T_{++} \geq 0.
\end{equation}
Therefore, \begin{equation}\label{103-tw-nv} \sigma_\infty = \sigma_c(0) < \sigma_c(\sqrt{Y_{++}}) < \sigma_c(\infty) = 1 \qquad \forall \  Y_{++} > 0. \end{equation}

Now to construct a non-negative solution to \eqref{101-tw-nv} passing through $ (R_{++}, Y_{++}) = (0,0) $, let $ \varepsilon > 0 $ and consider the approximating initial data $ Y_{++}(0) = \varepsilon > 0 $. Denote the solution as $ Y_{++}(R_{++};\varepsilon) $. By the Picard-Lindel\"of theorem, $ Y_{++}(R_{++};\varepsilon) $ exists at least for small $ R_{++} $. 

We claim that $ Y_{++}(R_{++};\varepsilon) >0 $ as long as the solution exists for $ R_{++} > 0 $. Otherwise, there will be a $ R_{++,0} $ within the existing interval, such that $ Y_{++}(R_{++,0};\varepsilon) =0 $ but $ Y_{++}(R_{++};\varepsilon) > 0 $ for $ 0 < R_{++} < R_{++,0} $. However, this contradicts to \eqref{101-tw-nv}, which yields that $ \frac{dY_{++}}{dR_{++}} \vert_{R_{++,0}} = 2 R_{++,0} > 0 $.

Thanks to \eqref{103-tw-nv}, solving \eqref{101-tw-nv} implies that
\begin{equation}
    \label{104-tw-nv}
    Y_{++}(R_{++};\varepsilon) \leq  R_{++}^2 + 4 R_{++} + \varepsilon.
\end{equation}
With a continuity argument, one can conclude that $ Y_{++}(R_{++};\varepsilon) $ exists for all $ R_{++}> 0 $, $ Y_{++}(R_{++};\varepsilon) > 0 $, and it is the unique solution to the approximating initial value problem. 

Send $ \varepsilon \rightarrow 0 $ in the sequence $ \lbrace Y_{++}(R_{++};\varepsilon) \rbrace_\varepsilon $. It is easy to verify that for arbitrary interval $ R_{++} \in \lbrack 0, \mathfrak R_{++} \rbrack $ for $ \mathfrak R_{++} \in (0,\infty) $. $ \lbrace Y_{++}(R_{++};\varepsilon) \rbrace_\varepsilon $ is a sequence of uniformly bounded and equicontinuous functions. Thus the  Arzela-Ascoli theorem implies that the limit
\begin{equation}
    \label{105-tw-nv}
    Y_{{++},0}(R_{++}) := \lim_{\varepsilon \rightarrow 0} Y_{++}(R_{++};\varepsilon)
\end{equation}
exists, and is bounded, continuous, and non-negative for $ R_{++} \in \lbrack 0, \mathfrak R_{++} \rbrack $. Since $ \mathfrak R_{++} $ is arbitrary, with a bootstrap argument, one can conclude that $ Y_{++,0} \geq 0 $ and $ Y_{++,0} $ is smooth for all $ R_{++}> 0 $.

Moreover, $ Y_{++,0} $ solves \eqref{101-tw-nv}, $ Y_{++,0}(0) = 0 $, and thus $ Y_{++,0} > 0 $ for all $ R_{++} > 0 $. In addition, thanks to \eqref{102-tw-nv}, 
one has that from \eqref{101-tw-nv}
\begin{equation}
    \label{106-tw-nv}
    \frac{dY_{++}}{dR_{++}} = 2\lbrack R_{++} - \underbrace{\frac{d\sigma_c(0)}{dT_{++}}}_{>0} \sqrt{Y_{++}} + \mathcal O(Y_{++}) \rbrack
\end{equation}
for $ Y_{++} $ small enough. Thus one can conclude that $ (R_{++}, Y_{++,0}(R_{++})) $ is the unique integral curve for \eqref{101-tw-nv} passing through $ (R_{++}, Y_{++}) = (0,0) $.

\smallskip 

Now we are ready to state the following characterization of the non-trivial traveling wave solutions: 
\begin{proposition}
\label{prop:tw++}
\begin{enumerate}
    \item Let $ (R_{++}, Y_{{++},0}(R_{++}) ) $ be the unique integral curve of \eqref{101-tw-nv} passing though $ (R_{++}, Y_{++}) = (0,0) $ constructed above. Then one has that $ Y_{++,0}(R_{++}) > 0, \ \forall \ R_{++} > 0 $.   

    \item 
    Consider the system of ODEs in \eqref{008-tw-nv} with $ (R_c, T_c)\vert_{\xi = 0} = (R_{++}, \sqrt{Y_{++,0}(R_{++})}) $ for arbitrary $ R_{++} > 0 $. 
    Then the initial value problem exists a global solution for all $ \xi \in \mathbb R $, satisfying
    \begin{itemize}
        \item $ R_c > 0 $ and $ T_c > 0 $ for all $ \xi \in \mathbb R $; 
        \item $ \lim_{\xi\rightarrow -\infty} (R_c, T_c) = 0 $;
        \item $ c_{1,++} e^{ c_{1,++} \xi} < R_c(\xi) , \ T_c(\xi) < c_{2,++} e^{ c_{2,++} \xi}, \ \forall \xi \in \mathbb R $, for some $ c_{1,++}, c_{2,++} \in (0,\infty) $. 
    \end{itemize}

    \item
    In particular, one can obtain the traveling wave solution $ (\sigma_c, S_c) $, satisfying \eqref{eq:tw-nv-ff}, by the formula \eqref{010-tw-nv} and $ S_c = R_c + \sigma_\infty $. 
    \end{enumerate}
    
\end{proposition}
\begin{proof}
The existence and uniqueness of the integral curve $ (R_{++}, Y_{{++},0}(R_{++}) ) $ follows from \eqref{101-tw-nv}--\eqref{106-tw-nv}. The existence of solutions to the initial value problem \eqref{008-tw-nv} follows easily from the integral curve. In particular, \eqref{014-tw-nv} yields the asymptotic behavior of $ (R_c, T_c) $. 
\end{proof}

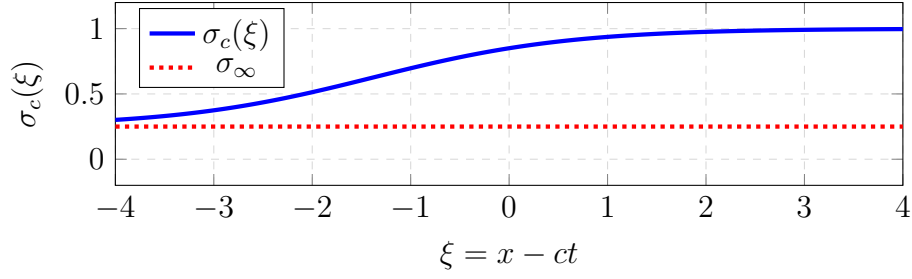
\begin{figure}[htbp]
\centering
\begin{tikzpicture}
\begin{axis} [
    xlabel={$\xi = x - ct$},
    ylabel={$\sigma_c(\xi)$},
    xmin=-4, xmax=4,
    ymin=-0.2, ymax=1.2,
    samples=200,
    grid=both,
    grid style={dashed, gray!30},
    width=12cm,
    height=4cm,
    legend pos=north west
]

\addplot[domain=-4:4, blue, ultra thick] {1/4 +  (3 * exp (x)) / (1 + 4 * exp(x))};
\addlegendentry{$\sigma_c(\xi)$}

\addplot[domain=-4:4, red, dotted, ultra thick] {1/4};
\addlegendentry{$\sigma_\infty$}

\end{axis}
\end{tikzpicture}
\caption{Sample of traveling wave profile $\sigma_c(\xi) > \sigma_\infty = 1/4, \ c > 0 $, from Proposition \ref{prop:tw++}}
\label{fig:sigma_c_tw++}
\end{figure}

The remaining cases follow with similar arguments. We will sketch the key steps in the following subsections. 

\subsubsection{The case when $ \sigma_c > \sigma_\infty $ and $ c < 0 $.} 
\label{subsubsec:+-}
Thanks to proposition \ref{prop:tw-nv-cnt}, we have $ T_c = \partial_\xi S_c < 0 $. Thus from \eqref{008-tw-nv}, one has 
\begin{equation}
    \label{201-tw-nv}
    R_c(\xi) = \int_{-\infty}^\xi T_c(\xi')\,d\xi' < 0.
\end{equation}
Moreover, $ p = p(x) $ from \eqref{006-tw-nv}, one can easily verify that
\begin{equation}
    \label{202-tw-nv}
    p(0) = - c \sigma_\infty > 0, \quad p(\sigma_\infty) > 0, \quad \text{and} \quad  p(1) = c(1-\sigma_\infty)  < 0.
\end{equation}
Since $ p $ has a negative leading order term, one has that 
\begin{equation}
    \label{203-tw-nv}
    \sigma_c = \sigma_c(T_c) : = \frac{(T_c - c) - \sqrt{(c-T_c)^2 + 4 c T_c \sigma_\infty}}{2 T_c} = \frac{-2 c \sigma_\infty}{(T_c-c) + \sqrt{(c-T_c)^2 + 4 c T_c \sigma_\infty}},
\end{equation}
and
\begin{equation}
    \label{204-tw-nv}
    \frac{d\sigma_c}{dT_c} = \sigma_c^2(T_c) \frac{ \sqrt{(c-T_c)^2 + 4c T_c \sigma_\infty} + (T_c - c) + 2 c \sigma_\infty }{2c\sigma_\infty \sqrt{(c-T_c)^2 + 4c T_c \sigma_\infty}} < 0,  \qquad \forall \ T_c \leq 0. 
\end{equation}
Thus, $ \sigma_\infty = \sigma_c(0) < \sigma_c(T_c) < \sigma_c(-\infty) = 1 $. 

Then correspondingly, one considers
\begin{equation}
    \label{205-tw-nv}
    \frac{d Y_{+-}}{d R_{+-}} = \frac{\partial_\xi T_{+-}^2}{\partial_\xi R_{+-}} = 2 \lbrack R_{+-} + \sigma_\infty - \sigma_c( - \sqrt{Y_{+-}}) )\rbrack, 
\end{equation}
where $ Y_{+-} : = T_{+-}^2 $ and $ \sigma_c $ given by \eqref{203-tw-nv}. With similar arguments as in section \ref{subsubsec:++}, one can show that there exists an integral curve $ (R_{+-}, Y_{+-,0}(R_{+-})) $ of \eqref{205-tw-nv} for all $ R_{+-} \leq 0 $,  passing through $ (R_{+-}, Y_{+-}) = (0,0) $, satisfying  $ Y_{+-,0} > 0 $ for all $ R_{+-} < 0 $. Thus one has the following proposition:
\begin{proposition}
    \label{prop:tw+-}
    \begin{enumerate}
    \item Let $ (R_{+-}, Y_{{+-},0}(R_{+-}) ) $ be the unique integral curve of \eqref{205-tw-nv} passing though $ (R_{+-}, Y_{+-}) = (0,0) $ as above. Then one has that $ Y_{+-,0}(R_{+-}) > 0, \ \forall \ R_{+-} < 0 $.   

    \item 
    Consider the system of ODEs in \eqref{008-tw-nv} with \eqref{203-tw-nv} and $ (R_c, T_c)\vert_{\xi = 0} = (R_{+-}, -\sqrt{Y_{+-,0}(R_{+-})}) $ for arbitrary $ R_{+-} < 0 $. 
    Then the initial value problem exists a global solution for all $ \xi \in \mathbb R $, satisfying
    \begin{itemize}
        \item $ R_c < 0 $ and $ T_c < 0 $ for all $ \xi \in \mathbb R $; 
        \item $ \lim_{\xi\rightarrow -\infty} (R_c, T_c) = 0 $;
        \item $ -c_{1,+-} e^{ c_{1,+-} \xi} < R_c(\xi) , \ T_c(\xi) < -c_{2,+-} e^{ c_{2,+-} \xi}, \ \forall \xi \in \mathbb R $, for some $ c_{1,+-}, c_{2,+-} \in (0,\infty) $. 
    \end{itemize}

    \item
    In particular, one can obtain the traveling wave solution $ (\sigma_c, S_c) $, satisfying \eqref{eq:tw-nv-ff}, by the formula \eqref{203-tw-nv} and $ S_c = R_c + \sigma_\infty $. 
    \end{enumerate}
\end{proposition}

\subsubsection{The case when $ \sigma_c < \sigma_\infty $ and $ c > 0 $.}\label{subsubsec:-+}

Thanks to proposition \ref{prop:tw-nv-cnt}, we have $ T_c = \partial_\xi S_c < 0 $. Thus from \eqref{008-tw-nv}, one has 
\begin{equation}
    \label{301-tw-nv}
    R_c(\xi) = \int_{-\infty}^\xi T_c(\xi')\,d\xi' < 0.
\end{equation}
Moreover, $ p = p(x) $ from \eqref{006-tw-nv}, one can easily verify that
\begin{equation}
    \label{302-tw-nv}
    p(0) = - c \sigma_\infty < 0, \quad p(\sigma_\infty) > 0, \quad \text{and} \quad  p(1) = c(1-\sigma_\infty)  > 0.
\end{equation}
Since $ p $ has a negative leading order term, one has that 
\begin{equation}
    \label{303-tw-nv}
    \sigma_c = \sigma_c(T_c) : = \frac{(T_c - c) + \sqrt{(c-T_c)^2 + 4 c T_c \sigma_\infty}}{2 T_c} = \frac{-2 c \sigma_\infty}{(T_c-c) - \sqrt{(c-T_c)^2 + 4 c T_c \sigma_\infty}},
\end{equation}
and
\begin{equation}
    \label{304-tw-nv}
    \frac{d\sigma_c}{dT_c} = \sigma_c^2(T_c) \frac{ \sqrt{(c-T_c)^2 + 4c T_c \sigma_\infty} - \lbrack (T_c - c) + 2 c \sigma_\infty \rbrack }{ 2c\sigma_\infty \sqrt{(c-T_c)^2 + 4c T_c \sigma_\infty}} > 0,  \qquad \forall \ T_c \leq 0. 
\end{equation}
Thus, $ \sigma_\infty = \sigma_c(0) > \sigma_c(T_c) > \sigma_c(-\infty) = 0 $. 

Then correspondingly, one considers
\begin{equation}
    \label{305-tw-nv}
    \frac{d Y_{-+}}{d R_{-+}} = \frac{\partial_\xi T_{-+}^2}{\partial_\xi R_{-+}} = 2 \lbrack R_{-+} + \sigma_\infty - \sigma_c( - \sqrt{Y_{-+}}) )\rbrack, 
\end{equation}
where $ Y_{-+} : = T_{-+}^2 $ and $ \sigma_c $ given by \eqref{303-tw-nv}. With similar arguments as in section \ref{subsubsec:++}, one can show that there exists an integral curve $ (R_{-+}, Y_{-+,0}(R_{-+})) $ of \eqref{205-tw-nv} for all $ R_{-+} \leq 0 $,  passing through $ (R_{-+}, Y_{-+}) = (0,0) $, satisfying  $ Y_{-+,0} > 0 $ for all $ R_{-+} < 0 $. Thus one has the following proposition:
\begin{proposition}
    \label{prop:tw-+}
    \begin{enumerate}
    \item Let $ (R_{-+}, Y_{{-+},0}(R_{-+}) ) $ be the unique integral curve of \eqref{305-tw-nv} passing though $ (R_{-+}, Y_{-+}) = (0,0) $ as above. Then one has that $ Y_{-+,0}(R_{-+}) > 0, \ \forall \ R_{-+} < 0 $.   

    \item 
    Consider the system of ODEs in \eqref{008-tw-nv} with \eqref{203-tw-nv} and $ (R_c, T_c)\vert_{\xi = 0} = (R_{-+}, -\sqrt{Y_{-+,0}(R_{-+})}) $ for arbitrary $ R_{-+} < 0 $. 
    Then the initial value problem exists a global solution for all $ \xi \in \mathbb R $, satisfying
    \begin{itemize}
        \item $ R_c < 0 $ and $ T_c < 0 $ for all $ \xi \in \mathbb R $; 
        \item $ \lim_{\xi\rightarrow -\infty} (R_c, T_c) = 0 $;
        \item $ -c_{1,-+} e^{ c_{1,-+} \xi} < R_c(\xi) , \ T_c(\xi) < -c_{2,-+} e^{ c_{2,-+} \xi}, \ \forall \xi \in \mathbb R $, for some $ c_{1,-+}, c_{2,-+} \in (0,\infty) $. 
    \end{itemize}

    \item
    In particular, one can obtain the traveling wave solution $ (\sigma_c, S_c) $, satisfying \eqref{eq:tw-nv-ff}, by the formula \eqref{303-tw-nv} and $ S_c = R_c + \sigma_\infty $. 
    \end{enumerate}
\end{proposition}

\subsubsection{The case when $ \sigma_c < \sigma_\infty $ and $ c < 0 $.}\label{subsubsec:--}

Thanks to proposition \ref{prop:tw-nv-cnt}, we have $ T_c = \partial_\xi S_c > 0 $. Thus from \eqref{008-tw-nv}, one has 
\begin{equation}
    \label{401-tw-nv}
    R_c(\xi) = \int_{-\infty}^\xi T_c(\xi')\,d\xi' > 0.
\end{equation}
Moreover, $ p = p(x) $ from \eqref{006-tw-nv}, one can easily verify that
\begin{equation}
    \label{402-tw-nv}
    p(0) = - c \sigma_\infty > 0, \quad p(\sigma_\infty) < 0, \quad \text{and} \quad  p(1) = c(1-\sigma_\infty)  < 0.
\end{equation}
Since $ p $ has a positive leading order term, one has that 
\begin{equation}
    \label{403-tw-nv}
    \sigma_c = \sigma_c(T_c) : = \frac{(T_c - c) - \sqrt{(c-T_c)^2 + 4 c T_c \sigma_\infty}}{2 T_c} = \frac{-2 c \sigma_\infty}{(T_c-c) + \sqrt{(c-T_c)^2 + 4 c T_c \sigma_\infty}},
\end{equation}
and
\begin{equation}
    \label{404-tw-nv}
    \frac{d\sigma_c}{dT_c} = \sigma_c^2(T_c) \frac{ \sqrt{(c-T_c)^2 + 4c T_c \sigma_\infty} + \lbrack (T_c - c) + 2 c \sigma_\infty \rbrack }{ 2c\sigma_\infty \sqrt{(c-T_c)^2 + 4c T_c \sigma_\infty}} < 0,  \qquad \forall \ T_c \geq 0. 
\end{equation}
Thus, $ \sigma_\infty = \sigma_c(0) > \sigma_c(T_c) > \sigma_c(\infty) = 0 $. 

Then correspondingly, one considers
\begin{equation}
    \label{405-tw-nv}
    \frac{d Y_{--}}{d R_{--}} = \frac{\partial_\xi T_{--}^2}{\partial_\xi R_{--}} = 2 \lbrack R_{--} + \sigma_\infty - \sigma_c( \sqrt{Y_{--}}) )\rbrack, 
\end{equation}
where $ Y_{--} : = T_{--}^2 $ and $ \sigma_c $ given by \eqref{403-tw-nv}. With similar arguments as in section \ref{subsubsec:++}, one can show that there exists an integral curve $ (R_{--}, Y_{--,0}(R_{--})) $ of \eqref{405-tw-nv} for all $ R_{--} \geq 0 $,  passing through $ (R_{--}, Y_{--}) = (0,0) $, satisfying  $ Y_{--,0} > 0 $ for all $ R_{--} > 0 $. Thus one has the following proposition:
\begin{proposition}
    \label{prop:tw--}
    \begin{enumerate}
    \item Let $ (R_{--}, Y_{{--},0}(R_{--}) ) $ be the unique integral curve of \eqref{405-tw-nv} passing though $ (R_{--}, Y_{--}) = (0,0) $ as above. Then one has that $ Y_{--,0}(R_{--}) > 0, \ \forall \ R_{--} > 0 $.   

    \item 
    Consider the system of ODEs in \eqref{008-tw-nv} with \eqref{203-tw-nv} and $ (R_c, T_c)\vert_{\xi = 0} = (R_{--}, \sqrt{Y_{--,0}(R_{--})}) $ for arbitrary $ R_{--} > 0 $. 
    Then the initial value problem exists a global solution for all $ \xi \in \mathbb R $, satisfying
    \begin{itemize}
        \item $ R_c > 0 $ and $ T_c > 0 $ for all $ \xi \in \mathbb R $; 
        \item $ \lim_{\xi\rightarrow -\infty} (R_c, T_c) = 0 $;
        \item $ c_{1,--} e^{ c_{1,--} \xi} < R_c(\xi) , \ T_c(\xi) < c_{2,--} e^{ c_{2,--} \xi}, \ \forall \xi \in \mathbb R $, for some $ c_{1,--}, c_{2,--} \in (0,\infty) $. 
    \end{itemize}

    \item
    In particular, one can obtain the traveling wave solution $ (\sigma_c, S_c) $, satisfying \eqref{eq:tw-nv-ff}, by the formula \eqref{403-tw-nv} and $ S_c = R_c + \sigma_\infty $. 
    \end{enumerate}
\end{proposition}

\section{Conclusion}
In this paper, we have constructed a complete classification of stationary and traveling wave solutions to the one-dimensional hyperbolic Keller-Segel system with quorum sensitivity. For vacuum far field, we have shown that stationary solutions exist as a one-parameter family of piecewise smooth profiles with a single jump from vacuum to saturation. Traveling wave solutions in this case include both continuous profiles and those with a single discontinuity, with the entropy condition ruling out jumps between non-vacuum states and determining the admissible jump direction. For non-vacuum far field, we have established that stationary solutions exist only with a single jump connecting the prescribed far field state to either vacuum or saturation, while all traveling wave profiles are necessarily continuous and are characterized by four distinct families depending on the relative size of the profile and the far field value. These results provide explicit examples of the non-uniqueness of entropy-admissible weak solutions and highlight the limitations of the entropy condition as a selection criterion.

\bibliographystyle{plain}
\bibliography{references}

\end{document}